\documentclass[a4paper,reqno]{amsart}

\providecommand{\isSn}{0}
\providecommand{\isBirk}{0}

\newcommand{\ifSn}[2]{\ifthenelse{\equal{\isSn}{1}}{#1}{#2}}
\newcommand{\ifBirk}[2]{\ifthenelse{\equal{\isBirk}{1}}{#1}{#2}}

\usepackage[utf8]{inputenc}
\usepackage[english]{babel}
\usepackage{microtype} 

\usepackage{hyperref} 
\hypersetup{hidelinks}
\usepackage[nameinlink]{cleveref} 

\usepackage{orcidlink}

\usepackage{amsmath,amssymb,amsthm}
\usepackage{thmtools,thm-restate} 
\usepackage{mathtools} 
\usepackage{upgreek}

\ifnum\isSn=0 \usepackage{tikz}\fi
\usetikzlibrary{cd}
\usetikzlibrary{calc}

\usepackage{dsfont} 
\usepackage{pifont} 
\usepackage{nicefrac} 

\usepackage{ifthen}
\usepackage{enumitem} 
\setlist[itemize,1]{leftmargin=20pt}

\usepackage{framed}
\usepackage{subcaption}

\usepackage{subfiles}

\newcounter{i} 
\newtoks\striche 

\newcommand{\R}{\mathbb{R}}
\newcommand{\N}{\mathbb{N}} 
\newcommand{\Z}{\mathbb{Z}}

\newcommand{\C}{\mathbb{C}}

\newcommand{\mc}{\mathcal}

\newcommand{\Lp}[1]{\mathrm{L}^{#1}} 

\newcommand{\diffd}{\mathrm{d}} 
\newcommand{\dx}[1][x]{\,\diffd#1}

\newcommand{\cl}[2][]{\overline{#2}\ifthenelse{ \equal{#1}{} }{}{^{#1}}} 
\newcommand{\conj}[1]{\overline{#1}} 
\newcommand{\quspace}[3][]{{\raisebox{.2em}{$#1#2$}\mspace{-4.5mu}\left/\mspace{2mu}\raisebox{-.2em}{$#1#3$}\right.}} 

\providecommand{\complement}{\mathsf{c}}

\DeclareMathOperator{\ran}{ran}
\DeclareMathOperator{\mul}{mul}

\DeclareMathOperator{\id}{id}

\DeclareMathOperator{\dom}{dom}

\DeclareMathOperator{\Div}{div}
\renewcommand{\div}{\Div}
\DeclareMathOperator{\grad}{grad}
\DeclareMathOperator{\rot}{curl}

\DeclarePairedDelimiter{\sset}{\{}{\}} 
\DeclarePairedDelimiter{\norm}{\lVert}{\rVert}
\DeclarePairedDelimiter{\abs}{\vert}{\vert}

\DeclarePairedDelimiterX{\dset}[2]{\{}{\}}{#1\,\delimsize\vert\,\mathopen{} #2} 
\DeclarePairedDelimiterX{\scprod}[2]{\langle}{\rangle}{#1,#2}
\DeclarePairedDelimiterX{\dualprod}[2]{\langle}{\rangle}{#1,#2}
\DeclarePairedDelimiterX{\sdprod}[2]{\llangle}{\rrangle}{#1,#2} 

\renewcommand{\Re}{\operatorname{Re}}

\newcommand{\dualsymb}{\prime}
\newcommand{\adjunsymb}{\ast} 
\newcommand{\dadjunsymb}{\ast_{\mathrm{d}}} 
\newcommand{\hadjunsymb}{\ast_{\mathrm{h}}} 
\newcommand{\adjunX}[1]{^{\adjunsymb_{#1}}}

\newcommand{\adjun}[1][1]{%
  \setcounter{i}{1}%
  \striche={\adjunsymb}%
  \loop%
  \ifnum\value{i}<#1%
  \striche=\expandafter{\the\expandafter\striche\adjunsymb}%
  \stepcounter{i}%
  \repeat%
  ^{\the\striche}%
}

\newcommand{\dadjun}[1][1]{%
  \setcounter{i}{1}%
  \striche={\dadjunsymb}%
  \loop%
  \ifnum\value{i}<#1%
  \striche=\expandafter{\the\expandafter\striche\dadjunsymb}%
  \stepcounter{i}%
  \repeat%
  ^{\the\striche}%
}

\newcommand{\hadjun}[1][1]{%
  \setcounter{i}{1}%
  \striche={\hadjunsymb}%
  \loop%
  \ifnum\value{i}<#1%
  \striche=\expandafter{\the\expandafter\striche\hadjunsymb}%
  \stepcounter{i}%
  \repeat%
  ^{\the\striche}%
}

\newcommand{\dual}[1][1]{%
  \setcounter{i}{1}%
  \striche={\dualsymb}%
  \loop%
  \ifnum\value{i}<#1%
  \striche=\expandafter{\the\expandafter\striche\dualsymb}%
  \stepcounter{i}%
  \repeat%
  ^{\the\striche}%
}

\newcommand{\mapping}[4]{%
  \left\{%
    \begin{array}{rcl}%
      #1 &\to & #2, \\
      #3 &\mapsto & #4
    \end{array}%
  \right.%
}

\newcommand{\trans}{^{\mathsf{T}}}

\newcommand{\Hspace}{\mathrm{H}}
\newcommand{\idop}{\mathrm{I}}
\newcommand{\opid}{\idop}

\newcommand{\conC}{\mathrm{C}}

\newcommand{\hssymbol}{\mathcal{X}}
\newcommand{\hs}{\hssymbol_{+}}
\newcommand{\hsmid}{\hssymbol_{0}}
\newcommand{\hsdual}{\hssymbol_{-}}

\newcommand{\Zm}{\mathcal{Z}_{-}}
\newcommand{\Zp}{\mathcal{Z}_{+}}

\newcommand{\boundtr}{\gamma_{0}}
\newcommand{\normaltr}{\gamma_{\nu}}

\ifSn{%
  \theoremstyle{thmstyleone}%
}{%
  \theoremstyle{plain}
}
\newtheorem{theorem}{Theorem}[section]
\newtheorem{lemma}[theorem]{Lemma}
\newtheorem{proposition}[theorem]{Proposition}
\newtheorem{corollary}[theorem]{Corollary}

\ifSn{%
  \theoremstyle{thmstyletwo}%
}{%
  \theoremstyle{definition}%
}
\newtheorem{definition}[theorem]{Definition}
\newtheorem{example}[theorem]{Example}

\newtheorem{conjecture}[theorem]{Conjecture}

\Crefname{conjecture}{Conjecture}{Conjecture}
\crefname{conjecture}{conjecture}{conjecture}

\ifSn{%
  \theoremstyle{thmstylethree}%
}{%
  \theoremstyle{remark}%
}
\newtheorem{remark}[theorem]{Remark}

\if\isSn0 \fi 

\begin{document}


\title{Quasi Gelfand triples}

\ifSn{
  \author*{\fnm{Nathanael} \sur{Skrepek}}
  \email{nathanael.skrepek@math.tu-freiberg.de}

  \affil*{\orgdiv{Department}, \orgname{Organization}, \orgaddress{\street{Street}, \city{City}, \postcode{100190}, \state{State}, \country{Country}}}
}%
{
  \author[Nathanael Skrepek]{Nathanael Skrepek\,\orcidlink{0000-0002-3096-4818}}
  \thanks{\textit{E-mail}: \href{mailto:nathanael.skrepek@math.tu-freiberg.de}{nathanael.skrepek@math.tu-freiberg.de}}
}
\date{\today}

%
\ifSn%
{}%
{}

\ifSn{}{
\address{TU Bergakademie Freiberg \\
  Institute of Applied Analysis \\
  Akademiestraße 6 \\
  D-09596 Freiberg \\
  Germany}
\email{nathanael.skrepek@math.tu-freiberg.de}
}

\begin{abstract}
  We generalize the notion of Gelfand triples (also called Banach-Gelfand triples or rigged Hilbert spaces) by dropping the necessity of a continuous embedding. This means in our setting we lack of a chain inclusion. We replace the continuous embedding by a closed embedding of a dense subspace. This new notion will be called \emph{quasi Gelfand triple}.
  These triples appear naturally, when we regard the boundary spaces of spatially multidimensional differential operators, e.g., the Maxwell operator.
  We will show that there is a smallest space where we can continuously embed the entire triple. Moreover, we will show density results for intersections of members of the quasi Gelfand triple.
  Finally, we show that every quasi Gelfand triple can be decomposed into two ``ordinary'' Gelfand triples.
\end{abstract}

\maketitle 

\section{Introduction}


Normally, when we talk about Gelfand triples we have a Hilbert space $\hsmid$ and
a reflexive Banach space $\hs$ that can be continuously and densely embedded into
$\hsmid$. The third space $\hsdual$ is given by the completion of $\hsmid$ with respect to
\begin{align*}
  \norm{g}_{\hsdual}
  \coloneqq
  \sup_{f\in\hs\setminus\sset{0}} \frac{\abs{\scprod{g}{f}_{\hsmid}}}{\norm{f}_{\hs}}.
\end{align*}
The duality between $\hs$ and $\hsdual$ is given by
\[
  \dualprod{g}{f}_{\hsdual,\hs} = \lim_{k\to\infty}\scprod{g_{k}}{f}_{\hsmid},
\]
where $(g_{k})_{k\in\N}$ is a sequence in $\hsmid$ that converges to $g$ in $\hsdual$.
The space $\hsdual$ is then isometrically isomorphic to $\hs\dual$.
The theory of Gelfand triples was introduced by I.M.\ Gelfand and A.G.\ Kostyuchenko~\cite{GeKo1955}. The concept has been refined over time.
In the introduction of~\cite{CoGh15} they give a short historical overview of Gelfand triples.
We want to weaken the assumptions such that the norm of $\hs$ is not
necessarily related to the norm of $\hsmid$. Hence, we cannot expect a continuous embedding of $\hs$ into $\hsmid$.

In~\cite{Sk21} this weakened concept was introduced for the case, where $\hs$ is also a Hilbert space.
Here we will introduce these generalized triples also for reflexive Banach spaces $\hs$. Moreover, in~\cite{Sk21} the notion of \emph{quasi Gelfand triples} was developed with the focus on a solution theory for port-Hamiltonian systems.
In particular to handle the boundary spaces of differential operators, like, e.g., the Maxwell operator.
In this work we will take a wider and deeper look at quasi Gelfand triples.

Even a little bit earlier also~\cite{CoGh15} introduced the notion of such quasi Gelfand triples for Hilbert spaces. They call it \emph{triples of closely embedded Hilbert spaces}. The motivation in~\cite{CoGh15} were weighted Sobolev and $\Lp{2}$ spaces, were the positive weight is neither bounded from above nor from below.

In this work we will lift the setting of~\cite{Sk21} to Banach spaces. So the beginning will be relatively similar to the introduction of quasi Gelfand triples in~\cite{Sk21}. This lifting has also be done in the Ph.D.\ thesis~\cite{Sk-Phd}. However, we go beyond the refinements of~\cite{Sk-Phd} and show that there exists a smallest space where we can structure preservingly embed the entire quasi Gelfand triple in \Cref{sec:quasi-Gelfand-triples-with-Hilbert-spaces}. Furthermore, we show a bijective relation between quasi Gelfand triples and Gram operators in \Cref{sec:gram-operator}.
This connection to Gram operators has also been discovered in~\cite{CoGh15} or it was actually the starting point of their journey. They call the Gram operator \emph{the Hamiltonian} of the triple. However, we take the next step and utilize this connection to the Gram operator to construct a decomposition of the quasi Gelfand triple into two ``ordinary'' Gelfand triples.

In \cite{BuCoSh02} the authors construct suitable boundary spaces for the tangential trace and the twisted tangential trace that correspond to the $\rot$ operator. These spaces naturally form a quasi Gelfand triple with $\Lp{2}(\partial\Omega)$ as pivot space. However, they did not pay a lot of attention to this additional structure as they develop their theory particular for the $\Hspace(\rot,\Omega)$ traces (tangential and twisted tangential trace). Moreover, they also give an explicit decomposition of the quasi Gelfand triple into two ``ordinary'' Gelfand triple (without calling it that).

In \Cref{sec:motivation} we will bring up the setting of \cite{BuCoSh02} as a motivation for the notion of quasi Gelfand triple. However, it is also suitable for other pairs of differential operators, e.g., $(\operatorname{symCurl},\operatorname{Curl})$, $(\operatorname{CurlCurl}\trans, \operatorname{CurlCurl}\trans)$, $(\operatorname{symGrad},\operatorname{Div})$, etc..

There is also a link to the notion of \emph{quasi boundary triples}, which was introduced in \cite{BeLa07}. The combination of boundary triples and quasi Gelfand triples is not entirely the same as quasi boundary triples, however both can be used to overcome limitations of boundary triples alone.






%

\section{Preliminary}


Since we will often switch between Hilbert space inner products and dual pairings, it is more convenient to always regard the anti-dual space instead of the dual space, which we will do. The anti-dual space is the space of all continuous antilinear mappings from the vectors space to $\C$.
Moreover, we will use a generalized concept for (unbounded) linear operators, namely linear relations.
The following notion of linear relations, dual pairs and adjoints with respect to dual pairs are carefully covered in~\cite[Ch.~1, Ch.~2]{Sk-Phd}.

A linear relation $T$ between the vector spaces $X$ and $Y$ is a linear subspace
of $X\times Y$. Clearly, every linear operator is also a linear relation
(we do not distinguish between a function and its graph).
For linear operators we have
\(
\begin{bsmallmatrix}
  x \\ y
\end{bsmallmatrix}
\in T
\)
is equivalent to $Tx = y$.
We will use the following notation
\begin{align*}
  \ker T &\coloneqq \dset*{x \in X}{\begin{bsmallmatrix} x \\ 0\end{bsmallmatrix} \in T},
         &\ran T &\coloneqq \dset*{y \in Y}{\exists x :\begin{bsmallmatrix} x \\ y\end{bsmallmatrix} \in T},
  \\
  \mul T &\coloneqq \dset*{y \in Y}{\begin{bsmallmatrix}0 \\ y\end{bsmallmatrix} \in T},
         &\dom T &\coloneqq \dset*{x \in X}{\exists y :\begin{bsmallmatrix}x \\ y\end{bsmallmatrix} \in T}.
\end{align*}
Thus, $T$ is single-valued (an operator), if $\mul T = \sset{0}$.
The closure $\cl{T}$ of a linear relation $T$ is the closure in $X\times Y$.
Note that every linear relation is closable. Also every operator has a
closure as a linear relation, but its closure can be multi-valued. Therefore,
showing $\mul \cl{T} = \sset{0}$ is necessary, even if $\mul T = \sset{0}$.

\begin{definition}\label{def:dual-pair}
  Let $X$ and $Y$ be Banach spaces and let $\dualprod{\cdot}{\cdot}_{Y,X}\colon Y \times X \to \C$ be continuous and sesquilinear (linear in the first argument and antilinear in the second argument). We define
  \begin{align*}
    \Psi\colon\mapping{Y}{X\dual}{y}{\dualprod{y}{\cdot}_{Y,X},}
    \quad\text{and}\quad
    \Phi\colon\mapping{X}{Y\dual}{x}{\conj{\dualprod{\cdot}{x}_{Y,X}}.}
  \end{align*}
  If $\Psi$ is isometric and bijective, then we say that $(X,Y)$ is a \emph{(anti-)dual pair} and $\dualprod{\cdot}{\cdot}_{Y,X}$ is its \emph{(anti-)dual pairing}.

  We define
  \begin{equation*}
    \dualprod{x}{y}_{X,Y} \coloneqq \conj{\dualprod{y}{x}_{Y,X}},
  \end{equation*}
  which is again a sesquilinear form.

  If also $\Phi$ is isometric and bijective, then we say that $(X,Y)$ is a \emph{complete (anti-)\-dual pair}.
\end{definition}

Clearly, $(X,X\dual)$ is a dual pair with the canonical dual pairing $\dualprod{x'}{x}_{X\dual,X} = x'(x)$ and it is complete, if $X$ is reflexive. For a Hilbert space $(H,H)$ is a complete dual pair with the inner product as dual pairing $\dualprod{x}{y}_{H,H} = \scprod{x}{y}_{H}$. However, if we think of the Sobolev space $\Hspace^{1}(\R)$ there are two ``natural'' possible dual pairings: the standard Hilbert space (complete) dual pair $(\Hspace^{1}(\R),\Hspace^{1}(\R))$ and the duality induced by the $\Lp{2}$ inner product $(\Hspace^{1}(\R),\Hspace^{-1}(\R))$ given by $\dualprod{x}{y}_{\Hspace^{1}(\R),\Hspace^{-1}(\R)} = \lim_{n\to\infty} \scprod{x}{y_{n}}_{\Lp{2}(\R)}$. Hence, in order to avoid saying both $\Hspace^{1}(\R)$ and $\Hspace^{-1}(\R)$ is the dual space of $\Hspace^{1}(\R)$, which can lead to confusion, we prefer the term (complete) dual pair.

\begin{definition}\label{def:adjoint-relation}
  Let $(X_{1},Y_{1})$, $(X_{2},Y_{2})$ be dual pairs and $A$ a linear relation between $X_{1}$ and $X_{2}$. Then we define the \emph{adjoint linear relation} by
  \begin{equation*}
    A\adjunX{Y_{2}\times Y_{1}} \coloneqq
    \dset*{
      \begin{bmatrix}
        y_{2} \\ y_{1}
      \end{bmatrix}
      \in Y_{2} \times Y_{1}
    }{
      \dualprod{y_{2}}{x_{2}}_{Y_{2},X_{2}} = \dualprod{y_{1}}{x_{1}}_{Y_{1},X_{1}}
      \text{ for all }
      \begin{bmatrix}
        x_{1} \\ x_{2}
      \end{bmatrix}
      \in A
    }.
  \end{equation*}
\end{definition}

We will just write $A\adjun$, if the dual pairs are clear.

For a Banach space $X$, we will regard the dual pair $(X,X\dual)$ for the adjoint, if no other dual pair is given. Similar, for a Hilbert space $H$ we will regard the dual pair $(H,H)$, if no other dual pair is given.

Note that this definition matches the usual Hilbert space adjoint, if $A$ is a densely defined operator between two Hilbert spaces.

\begin{remark}
  If $A$ is an operator ($\mul A = \sset{0}$) from $X_{1}$ to $X_{2}$, then we can characterize the domain of $A\adjun$ by
  \begin{equation*}
    y_{2} \in \dom A\adjun \quad\Leftrightarrow\quad
    \dom A \ni x_{1} \mapsto \dualprod{y_{2}}{A x_{1}}_{Y_{2},X_{2}} \;\text{is continuous w.r.t.\ } \norm{\cdot}_{X_{1}}.
  \end{equation*}
  Moreover, we have the following relations
  \begin{equation*}
    \ker A\adjun = (\ran A)^{\perp} \quad\text{and}\quad \mul A\adjun = (\dom A)^{\perp},
  \end{equation*}
  where $M^{\perp}$ denotes the annihilator space of $M$ (which is the orthogonal complement in the Hilbert space case).
\end{remark}

\section{Motivation}%
\label{sec:motivation}

Let $\Omega \subseteq \R^{3}$ be a bounded open set with bounded Lipschitz boundary. For $f, g \in \conC^{\infty}(\R^{3})$ we have the following integration by parts formula:
\begin{equation*}
  \scprod{\div f}{g}_{\Lp{2}(\Omega)} + \scprod{f}{\grad g}_{\Lp{2}(\Omega)} = \scprod{\nu \cdot \boundtr f}{\boundtr g}_{\Lp{2}(\partial\Omega)},
\end{equation*}
where $\nu$ is the normal vector on $\partial \Omega$ and $\boundtr$ is the boundary trace.
It is also well known that we can extend this formula for $f \in \Hspace(\div,\Omega)$ and $g \in \Hspace^{1}(\Omega)$:
\begin{equation*}
  \scprod{\div f}{g}_{\Lp{2}(\Omega)} + \scprod{f}{\grad g}_{\Lp{2}(\Omega)} = \dualprod{\normaltr f}{\boundtr g}_{\Hspace^{-\nicefrac{1}{2}}(\partial\Omega),\Hspace^{\nicefrac{1}{2}}(\partial\Omega)},
\end{equation*}
where $\normaltr$ is the continuous extension of $\nu \cdot \boundtr$. In this extension we stumble over the Gelfand triple $(\Hspace^{\nicefrac{1}{2}}(\partial\Omega),\Lp{2}(\partial\Omega),\Hspace^{-\nicefrac{1}{2}}(\partial\Omega))$. However, in general such an integration by parts formula will not automatically lead to such an extension where we can replace the $\Lp{2}$ inner product on the boundary by a dual pairing that comes from a Gelfand triple with $\Lp{2}(\partial\Omega)$ as pivot space. For example for $f,g \in \conC^{\infty}(\R^{3})$ we have
\begin{equation}\label{eq:int-by-parts-rot}
  \scprod{\rot f}{g}_{\Lp{2}(\Omega)} + \scprod{f}{\rot g}_{\Lp{2}(\Omega)} = \dualprod{\nu \times \boundtr f}{(\nu \times \boundtr g)\times \nu}_{\Lp{2}(\partial\Omega)},
\end{equation}
but contrary to the previous case neither $\nu\times \boundtr $ nor $(\nu \times \boundtr ) \times \nu$ can be continuously extended to $\Hspace(\rot,\Omega)$ such that its codomain is still $\Lp{2}(\partial\Omega)$ (or can be continuously embedded into $\Lp{2}(\partial\Omega)$), see~\cite[Ex.~A.4]{Sk21}.
Hence, in order to better understand the relation between the extension of \eqref{eq:int-by-parts-rot} to $\Hspace(\rot,\Omega)$ and the $\Lp{2}(\partial\Omega)$ inner product we need a more general tool than Gelfand triples.
In order to try to find a suitable boundary space such that we can extend $\nu \times \boundtr$ on $\Hspace(\rot,\Omega)$, we endow $\ran (\nu \times \boundtr)$ with the range norm that comes from $\Hspace(\rot,\Omega)$. This gives a norm on a dense subspace of $\Lp{2}_{\tau}(\partial\Omega) = \dset{\phi \in \Lp{2}(\partial\Omega)}{\nu \cdot f = 0}$ that is unrelated to $\norm{\cdot}_{\Lp{2}(\partial\Omega)}$. This setting will be our starting point.
This particular problem was treated in \cite{Sk21}.
Here we want to discover the world of quasi Gelfand triples without any particular applications in mind (or maybe with \Cref{con:weak,con:strong} in mind).


\subsection{Starting point}
We will have the following setting: Let
$\hsmid$\label{symb:hsmid} be a Hilbert space with the inner product $\scprod{\cdot}{\cdot}_{\hsmid}$ and $\scprod{\cdot}{\cdot}_{\hs}$ be another inner product on $\hsmid$ (not necessarily related to $\scprod{\cdot}{\cdot}_{\hsmid}$), which is defined on a dense (w.r.t.\ $\norm{\cdot}_{\hsmid}$) subspace $\tilde{D}_{+}$\label{symb:Dplus-tilde} of $\hsmid$.
We denote the completion of $\tilde{D}_{+}$ w.r.t.\
$\norm{\cdot}_{\hs}$ ($\norm{f}_{\hs} \coloneqq \sqrt{\scprod{f}{f}_{\hs}}$) by $\hs$. This completion is, by construction
a Hilbert space with the extension of $\scprod{\cdot}{\cdot}_{\hs}$, for which we use
the same symbol. Now we have $\tilde{D}_{+}$ is dense in $\hsmid$ w.r.t.\
$\norm{\cdot}_{\hsmid}$ and dense in $\hs$ w.r.t.\ $\norm{\cdot}_{\hs}$.
\Cref{fig:quasi-gelfand-motivation-setting} illustrates this setting.

\smallskip\par
Note that $\hs$, as a Hilbert space, is automatically reflexive.
For the further construction the crucial property of $\hs$ is its reflexivity.
Hence, we will weaken the previous setting such that $\hs$ is only a reflexive Banach space:
\begin{itemize}
  \item $\hsmid$ Hilbert space endowed with $\scprod{\cdot}{\cdot}_{\hsmid}$.

  \item $\tilde{D}_{+}$ dense subspace of $\hsmid$ (w.r.t.\
    $\norm{\cdot}_{\hsmid}$).

  \item $\norm{\cdot}_{\hs}$ another norm defined on $\tilde{D}_{+}$.

  \item $\hs$ completion of $\tilde{D}_{+}$ with respect to $\norm{\cdot}_{\hs}$ is reflexive.
\end{itemize}

\begin{figure}[ht]
  \centering
  \begin{tikzpicture}
    \begin{scope}
      \clip (-0.5,-0.5) circle (1.5);
      \fill[blue,opacity=0.15] (0,0) circle (1.5);
    \end{scope}
    \draw[thick](0,0) circle [radius=1.5];
    \draw[thick,blue](-0.5,-0.5) circle [radius=1.5];
    \node[left,blue] at (-2,-0.5) {$\hs$};
    \node[above] at (0,1.5) {$\hsmid$};
    \node[blue] at (-0.25,-0.25) {$\tilde{D}_{+}$};
  \end{tikzpicture}
  \caption{\label{fig:quasi-gelfand-motivation-setting}Setting of $\hsmid$, $\tilde{D}_{+}$ and $\hs$.}
\end{figure}


\begin{example}
  Let $\hsmid = \ell^{2}(\Z \setminus \sset{0})$ with the standard inner product $\scprod{x}{y}_{\hsmid} = \sum_{n=1}^{\infty} x_{n}\conj{y_{n}}+ x_{-n}\conj{y_{-n}}$. We define the inner product
  \[
    \scprod{x}{y}_{\hs} \coloneqq \sum_{n=1}^{\infty} n^{2} x_{n}\conj{y_{n}}+ \frac{1}{n^{2}}x_{-n}\conj{y_{-n}}
  \]
  and the set $\tilde{D}_{+} \coloneqq \dset{f \in \hsmid}{\norm{f}_{\hs} < +\infty}$. Clearly, this inner product is well-defined on $\tilde{D}_{+}$. Let $e_{i}$ denote the sequence which is $1$ on the $i$-th position and $0$ elsewhere. Since $\dset{e_{i}}{i \in \Z\setminus\sset{0}}$ is a orthonormal basis of $\hsmid$ and contained in $\tilde{D}_{+}$, $\tilde{D}_{+}$ is dense in $\hsmid$. The sequence $\big(\sum_{i=1}^{n} e_{-i}\big)_{n\in\N}$ is a Cauchy sequence with respect to $\norm{\cdot}_{\hs}$, but not with respect to $\norm{\cdot}_{\hsmid}$.
\end{example}

\begin{definition}\label{def:Dminus}
  We define
  \begin{align*}
    \renewcommand{\quad}{\mspace{10mu}}
    \norm{g}_{\hsdual} \coloneqq
    \sup_{f\in \tilde{D}_{+}\setminus\sset{0}}
    \frac{\abs{\scprod{g}{f}_{\hsmid}}}{\norm{f}_{\hs}}
    \quad \text{for}\ g\in \hsmid
    \quad \text{and} \quad
    D_{-} \coloneqq \dset[\Big]{g \in \hsmid}{\norm{g}_{\hsdual} < +\infty}.
  \end{align*}
  \label{symb:Dminus}%
  We denote the completion of $D_{-}$ w.r.t.\ $\norm{\cdot}_{\hsdual}$ by $\hsdual$.
  We will also denote the extension of $\norm{\cdot}_{\hsdual}$ to $\hsdual$ by $\norm{\cdot}_{\hsdual}$.
\end{definition}

\begin{remark}\label{re:D-minus-identified-with-X-plus-dual}
By definition of $D_{-}$ we can identify every $g \in D_{-}$ with an element of
$\hs\dual$ by the continuous extension of
\begin{equation*}
  \psi_{g}\colon
  \mapping{D_{+}}{\C}{f}{\scprod{g}{f}_{\hsmid},}
\end{equation*}
on $\hs$. We denote this extension again by $\psi_{g}$.
By definition of $D_{-}$ we have
$\norm{\psi_{g}}_{\hs\dual} = \norm{g}_{\hsdual}$ for $g\in D_{-}$.
Hence, we can extend the isometry
\begin{equation*}
  \Psi\colon
  \mapping{D_{-}}{\hs\dual}{g}{\psi_{g},}
\end{equation*}
by continuity on $\hsdual$, this is extension is again denoted by $\Psi$. So $\hsdual$ can be seen as the closure of $D_{-}$ in $\hs\dual$.
\end{remark}

We can define a dual pairing between $\hs$ and $\hsdual$ by
\begin{equation*}
  \dualprod{g}{f}_{\hsdual,\hs} \coloneqq \dualprod{\Psi g}{f}_{\hs\dual,\hs}.
\end{equation*}
However, this does not necessarily make $(\hs,\hsdual)$ a dual pair in the sense of \Cref{def:dual-pair}, because we do not know whether $\Psi$ is surjective.

\begin{lemma}\label{le:Dminus-complete}
  $D_{-}$ is complete with respect to $\norm{g}_{\hsdual \cap \hsmid} \coloneqq \sqrt{\vphantom{\norm{g}^{2}}\smash{\norm{g}_{\hsmid}^{2} + \norm{g}_{\hsdual}^{2}}}$.
\end{lemma}

\begin{proof}
  Let $(g_{n})_{n\in\N}$ be a Cauchy sequence in $D_{-}$ with respect to $\norm{\cdot}_{\hsdual\cap\hsmid}$.
  Then $(g_{n})_{n\in\N}$ is a convergent sequence in $\hsmid$ (w.r.t.\
  $\norm{\cdot}_{\hsmid}$) and a Cauchy sequence in $D_{-}$ (w.r.t.\
  $\norm{\cdot}_{\hsdual}$). We denote the limit in $\hsmid$ by $g_{0}$.
  By definition of $\norm{\cdot}_{\hsdual}$
  we obtain for $f\in \tilde{D}_{+}$
  \begin{align*}
    \abs{\scprod{g_{0}}{f}_{\hsmid}} = \lim_{n\to\infty} \abs{\scprod{g_{n}}{f}_{\hsmid}}
    \leq \lim_{n\to\infty} \norm{g_{n}}_{\hsdual} \norm{f}_{\hs} \leq C \norm{f}_{\hs}
  \end{align*}
  and consequently $g_{0} \in D_{-}$.

  Let $\epsilon > 0$ be arbitrary. Since $(g_n)_{n\in\N}$ is a Cauchy sequence with respect to $\norm{\cdot}_{\hsdual}$, there is an $n_{0}\in \N$ such that for all $f\in \tilde{D}_{+}$ with $\norm{f}_{\hs} = 1$
  \begin{align*}
    \abs{\scprod{g_{n}-g_{m}}{f}_{\hsmid}} \leq \frac{\epsilon}{2}, \quad \text{if}\quad n,m \geq n_{0}
  \end{align*}
  holds true. Furthermore, for every $f \in \tilde{D}_{+}$ there exists an $m_{f}\geq n_{0}$ such that $\abs{\scprod{g_0 - g_{m_{f}}}{f}_{\hsmid}} \leq \frac{\epsilon \norm{f}_{\hs}}{2}$, because $g_m\to g_0$ w.r.t.\ $\norm{\cdot}_{\hsmid}$. This yields
  \begin{align*}
    \frac{\abs{\scprod{g_{0}-g_{n}}{f}_{\hsmid}}}{\norm{f}_{\hs}} \leq \frac{\abs{\scprod{g_{0} - g_{m_{f}}}{f}_{\hsmid}}}{\norm{f}_{\hs}} + \frac{\abs{\scprod{g_{m_{f}}-g_{n}}{f}_{\hsmid}}}{\norm{f}_{\hs}} \leq \epsilon,
    \quad \text{if}\quad n\geq n_{0}.
  \end{align*}
  Since the right-hand-side is independent of $f$, we obtain
  \begin{align*}
    \norm{g_{0}-g_{n}}_{\hsdual} = \sup_{f\in \tilde{D}_{+}\setminus\sset{0}} \frac{\abs{\scprod{g_{0}-g_{n}}{f}_{\hsmid}}}{\norm{f}_{\hs}} \leq \epsilon,
    \quad \text{if}\quad n\geq n_{0}.
  \end{align*}
  Hence, $g_{0}$ is also the limit of $(g_{n})_{n\in\N}$ with respect to $\norm{\cdot}_{\hsdual}$ and consequently the limit of  $(g_{n})_{n\in\N}$ with respect to $\norm{\cdot}_{\hsdual\cap\hsmid}$.
\end{proof}

Strictly speaking $\tilde{D}_{+}$ and $D_{-}$ are subsets of $\hsmid$, but sometimes we rather want to regard them as subsets of $\hs$ and $\hsdual$, respectively. Hence, introduce the following embedding mappings
\begin{equation*}
  \tilde{\iota}_{+}\colon
  \mapping{\tilde{D}_{+}\subseteq \hs}{\hsmid}{f}{f,}
  \quad\text{and} \quad
  \iota_{-}\colon
  \mapping{D_{-}\subseteq \hsdual}{\hsmid}{g}{g.}
\end{equation*}
This allows us to distinguish between $f \in \tilde{D}_{+}$ as element of $\hs$ and $\tilde{\iota}_{+}(f)$ as element of $\hsmid$, if necessary. Clearly, the same for $g \in D_{-}$.

\begin{lemma}\label{th:iotap-and-iotam}
  The embedding
  $\tilde{\iota}_{+}$
  is a
  densely defined operator with $\ran \tilde{\iota}_{+}$ is dense in $\hsmid$
  and $\ker \tilde{\iota}_{+} = \sset{0}$. Furthermore, the embedding
  $\iota_{-}$
  is closed and
  $\ker \iota_{-} = \sset{0}$.
\end{lemma}

\begin{proof}
  By assumption on $\tilde{D}_{+}$ and definition of $\hs$ the embedding
  $\tilde{\iota}_{+}$ is densely defined and has a dense range. Clearly,
  $\ker \tilde{\iota}_{+} = \sset{0}$ and $\ker \iota_{-} = \sset{0}$.
  By \Cref{le:Dminus-complete} $\iota_{-}$ is closed.
\end{proof}

\begin{lemma}\label{le:adjoint-embedding}
  Let $\tilde{\iota}_{+}\adjun = \tilde{\iota}_{+}\adjunX{\hsmid \times \hs\dual}$
  denote the adjoint relation (w.r.t.\ the dualities $(\hsmid,\hsmid)$ and $(\hs,\hs\dual)$) of $\tilde{\iota}_{+}$.
  Then $\tilde{\iota}_{+}\adjun$ is an operator (single-valued, i.e., $\mul \tilde{\iota}_{+}\adjun = \sset{0}$) and
  $\ker \tilde{\iota}_{+}\adjun = \sset{0}$.
  Its domain coincides with $D_{-}$ and
  $\tilde{\iota}_{+}\adjun\iota_{-}\colon D_{-}\subseteq \hsdual \to \hs\dual$ is
  isometric.

  If $\ker \cl{\tilde{\iota}_{+}} = \sset{0}$, then $\ran \tilde{\iota}_{+}\adjun$ is dense in $\hs\dual$.
\end{lemma}

\begin{proof}
  The density of the domain of $\tilde{\iota}_{+}$ yields $\mul \tilde{\iota}_{+}\adjun = (\dom \tilde{\iota}_{+})^{\perp} = \sset{0}$, and $\cl[\hsmid]{\ran \tilde{\iota}_{+}} = \hsmid$ yields $\ker \tilde{\iota}_{+}\adjun = \sset{0}$.
  The following equivalences show $\dom \tilde{\iota}_{+}\adjun = D_{-}$:
\begin{align*}
  g \in \dom \tilde{\iota}_{+}\adjun
  \quad &\Leftrightarrow\quad \tilde{D}_{+} \ni f \mapsto \scprod{g}{\tilde{\iota}_{+} f}_{\hsmid} \; \text{is continuous w.r.t.}\ \norm{\cdot}_{\hs} \\
        &\Leftrightarrow\quad \sup_{f\in \tilde{D}_{+}\setminus\sset{0}} \frac{\abs{\scprod{g}{f}_{\hsmid}}}{\norm{f}_{\hs}} < +\infty \\
        &\Leftrightarrow\quad g \in D_{-}.
\end{align*}
For $g\in D_{-}\subseteq \hsdual$ we have
\begin{align*}
  \norm{g}_{\hsdual}
  = \sup_{f\in \tilde{D}_{+}\setminus\sset{0}} \frac{\abs{\scprod{\iota_{-} g}{f}_{\hsmid}}}{\norm{f}_{\hs}}
  = \sup_{f\in \tilde{D}_{+}\setminus\sset{0}} \frac{\abs{\dualprod{\tilde{\iota}_{+}\adjun \iota_{-}g}{f}_{\hs\dual,\hs}}}{\norm{f}_{\hs}}
  = \norm{\tilde{\iota}_{+}\adjun \iota_{-}g}_{\hs\dual},
\end{align*}
which proves that $\tilde{\iota}_{+}\adjun \iota_{-}$ is isometric.

Note that the reflexivity of $\hs$ implies $\cl{\tilde{\iota}_{+}} = \tilde{\iota}_{+}\adjun[2]$.
If $\ker \cl{\tilde{\iota}_{+}} = \sset{0}$, then the following equation implies the density of $\ran \tilde{\iota}_{+}\adjun$ in $\hs\dual$
\begin{equation*}
  \sset{0} = \ker \cl{\tilde{\iota}_{+}} = \ker \tilde{\iota}_{+}\adjun[2] = (\ran \tilde{\iota}_{+}\adjun)^{\perp}.
  \ifSn{\tag*{\qedhere}}{\qedhere}
\end{equation*}
\end{proof}

\begin{remark}\label{re:D-minus-iota-plus-adjun-X-plus-dual}
  As mentioned in \Cref{re:D-minus-identified-with-X-plus-dual} every
  $g\in D_{-}$ can be regarded as an element of $\hs\dual$ by $\psi_{g}$.
  Let $g\in D_{-}$, $f \in \hs$ and $(f_{n})_{n\in\N}$ in $\tilde{D}_{+}$ converging to $f$ w.r.t.\ $\norm{\cdot}_{\hs}$.
  Since $D_{-} = \dom \tilde{\iota}_{+}\adjun$, we have
  \begin{equation*}
    \dualprod{\psi_{g}}{f}_{\hs\dual,\hs} = \lim_{n\to\infty} \scprod{g}{f_{n}}_{\hsmid} = \lim_{n\to\infty} \scprod{\iota_{-} g}{\tilde{\iota}_{+} f_{n}}_{\hsmid}
    = \dualprod{\tilde{\iota}_{+}\adjun \iota_{-} g}{f}_{\hs\dual,\hs}
  \end{equation*}
  and consequently $\psi_{g} = \tilde{\iota}_{+}\adjun\iota_{-} g$.
  Hence, $\Psi D_{-} = \tilde{\iota}_{+}\adjun\iota_{-} D_{-} = \ran \tilde{\iota}_{+}\adjun$.
\end{remark}

\begin{proposition}\label{prop:iota-closed-equivalences}
  The following assertions are equivalent.
  \begin{enumerate}[label = \textrm{\textup{(\roman*)}}]
    \item\label{item:iota-closed-equivalences-i}
      There is a Hausdorff topological vector space $(Z,\mathcal{T})$ and two continuous
      embeddings $\phi_{\hs}\colon \hs \to Z$ and $\phi_{\hsmid}\colon \hsmid \to Z$ such
      that the diagram
      \[
          \begin{tikzcd}[column sep=normal, row sep=small]
            \tilde{D}_{+}\ar[shift right,swap]{dd}{\tilde{\iota}_{+}}\ar{r}{\id}& \hs\arrow[dr,"\phi_{\hs}"] & \\
            & & Z \\
            \tilde{D}_{+}\ar[shift right,swap]{uu}{\tilde{\iota}_{+}^{-1}}\ar{r}{\id}& \hsmid\arrow[ur,"\phi_{\hsmid}"'] &
          \end{tikzcd}
      \]
      commutes.

    \item\label{item:iota-closed-equivalences-ii}
      If $\tilde{D}_{+}\ni f_n \to 0$ w.r.t.\ $\norm{\cdot}_{\hs}$ and
      $\lim_{n\to\infty} f_n$ exists w.r.t.\ $\norm{\cdot}_{\hsmid}$, then this limit
      is also $0$ and if $\tilde{D}_{+}\ni f_n \to 0$ w.r.t.\
      $\norm{\cdot}_{\hsmid}$ and $\lim_{n\to\infty} f_n$ exists w.r.t.\
      $\norm{\cdot}_{\hs}$, then this limit is also $0$.

    \item\label{item:iota-closed-equivalences-iii}
      $\tilde{\iota}_{+} \colon \tilde{D}_{+} \subseteq \hs \to \hsmid, f \mapsto f$
      is closable (as an operator) and its closure is injective.

    \item\label{item:iota-closed-equivalences-iv}
      $D_{-}$ is dense in $\hsmid$ and dense in $\hs\dual$, i.e., $\Psi D_{-}$ is dense in $\hs\dual$.
  \end{enumerate}
\end{proposition}

\begin{proof}
  We will follow the strategy $\ref{item:iota-closed-equivalences-i} \Rightarrow \ref{item:iota-closed-equivalences-ii} \Rightarrow \ref{item:iota-closed-equivalences-iii} \Rightarrow \ref{item:iota-closed-equivalences-iv} \Rightarrow \ref{item:iota-closed-equivalences-i}$.
  \begin{itemize}[leftmargin=\parindent,itemindent=*,labelsep=3pt,parsep=2pt,itemsep=3pt,align=left]
    \item[$\ref{item:iota-closed-equivalences-i}\Rightarrow\ref{item:iota-closed-equivalences-ii}$:]
          Let $(f_n)_{n\in\N}$ be a sequence in $\tilde{D}_{+}$ such that $f_n \to \hat{f}$ w.r.t.\ $\hs$ and $f_n \to f$ w.r.t.\ $\hsmid$. Since $\mc T$ is coarser than both of the topologies induced by these norms, we also have
          \[
          \begin{tikzcd}[row sep=0.1em]
            & \hat{f}\\
            f_n \arrow{ur}{\mc T}\arrow{dr}{\mc T} & \\
            & f
          \end{tikzcd}
          \quad \text{in } Z.
          \]
          Since $\mc T$ is Hausdorff, we conclude $f=\hat{f}$.
          Hence, if either $\hat{f}$ or $f$ is $0$, then also the other is $0$.

    \item[$\ref{item:iota-closed-equivalences-ii}\Rightarrow\ref{item:iota-closed-equivalences-iii}$:]
          If $(f_n,f_n)_{n\in\N}$ is a
          sequence in $\tilde{\iota}_{+}$ that converges to $(0,f)\in\hs\times \hsmid$,
          then $f = 0$ by \textup{(ii)}. Hence, $\mul \cl{\tilde{\iota}_{+}} = \sset{0}$
          and consequently $\tilde{\iota}_{+}$ is closable.
          On the other hand, if $(f_{n},f_{n})_{n\in\N}$ is a sequence in
          $\tilde{\iota}_{+}$ that converges to $(f,0)$, then $f=0$ by~\ref{item:iota-closed-equivalences-ii}.
          Consequently, $\ker \cl{\tilde{\iota}_{+}} = \sset{0}$ and
          $\cl{\tilde{\iota}_{+}}$ is injective.

    \item[$\ref{item:iota-closed-equivalences-iii}\Rightarrow\ref{item:iota-closed-equivalences-iv}$:]
          We have
          $(\dom \tilde{\iota}_{+}\adjun)^{\perp} = \mul \tilde{\iota}_{+}\adjun[2] = \mul \cl{\tilde{\iota}_{+}}$.
          Since $\tilde{\iota}_{+}$ is closable, we have
          $\mul \cl{\tilde{\iota}_{+}} = \sset{0}$, which implies that
          $\dom \tilde{\iota}_{+}\adjun$ is dense in $\hsmid$. By \Cref{le:adjoint-embedding}
          $\dom \tilde{\iota}_{+}\adjun$ coincides with $D_{-}$, which gives the density of $D_{-}$ in $\hsmid$.

          The second assertion of
          \Cref{le:adjoint-embedding} yields that
          $\ran \tilde{\iota}_{+}\adjun$ is
          dense in $\hs\dual$.
          By \Cref{re:D-minus-iota-plus-adjun-X-plus-dual} we have $\ran \tilde{\iota}_{+}\adjun = \Psi D_{-}$ and therefore the density of $\Psi D_{-}$ in $\hs\dual$.

    \item[$\ref{item:iota-closed-equivalences-iv}\Rightarrow\ref{item:iota-closed-equivalences-i}$:]
          Let $Y \coloneqq D_{-}$ be equipped with
          \begin{equation*}
            \norm{g}_{Y} \coloneqq \norm{g}_{\hsdual\cap\hsmid} = \sqrt{\norm{g}_{\hsdual}^{2} + \norm{g}_{\hsmid}^{2}}.
          \end{equation*}
          We define $Z \coloneqq Y\dual$ as the (anti-)dual space of $Y$. Then we have
          \begin{align*}
            \abs{\scprod{f}{g}_{\hsmid}}
            &\leq \norm{f}_{\hsmid} \norm{g}_{\hsmid} \leq \norm{f}_{\hsmid} \norm{g}_{Y}
            &&\mspace{-5mu}\text{for}\quad f\in \hsmid, g\in Y \\
            \text{and}\quad
            \abs{\dualprod{f}{\tilde{\iota}_{+}\adjun g}_{\hs,\hs\dual}}
            &\leq \norm{f}_{\hs} \underbrace{\norm{\tilde{\iota}_{+}\adjun g}_{\hs\dual}}_{=\mathrlap{\norm{g}_{\hsdual}}}
              \leq \norm{f}_{\hs}\norm{g}_{Y}
            &&\mspace{-5mu}\text{for}\quad f\in\hs, g\in Y.
          \end{align*}
          Hence, $\phi_{\hsmid}\colon f\mapsto \scprod{f}{\cdot}_{\hsmid}$ and $\phi_{\hs}\colon f \mapsto \dualprod{f}{\tilde{\iota}_{+}\adjun \cdot}_{\hs,\hs\dual}$ are continuous mappings from $\hsmid$ and $\hs$, respectively, into $Z$. The injectivity of these mappings follows from the density of $D_{-}$ in $\hsmid$ and $D_{-}$ in $\hs\dual$ ($\tilde{\iota}_{+}\adjun D_{-}$ dense in $\hs\dual$), respectively. For $f \in \tilde{D}_{+}$ we have
          \begin{align*}
            \phi_{\hs} f = \scprod{f}{\tilde{\iota}_{+}\adjun \cdot}_{\hs,\hs\dual} = \scprod{\tilde{\iota}_{+}f}{\cdot}_{\hsmid} = \phi_{\hsmid}\circ\tilde{\iota}_{+} f
          \end{align*}
          and consequently the diagram in~\ref{item:iota-closed-equivalences-i} commutes. \qedhere
  \end{itemize}
\end{proof}

If one and therefore all assertions in \Cref{prop:iota-closed-equivalences}
are satisfied, then $\hs \cap \hsmid$ is defined as the intersection in $Z$
and complete with the norm
$\norm{\cdot}_{{\hs\cap\hsmid}} \coloneqq \sqrt{\norm{\cdot}_{\hs}^{2} + \norm{\cdot}_{\hsmid}^{2}}$.
Moreover, we define $D_{+}$ as the closure of $\tilde{D}_{+}$ in $\hs \cap \hsmid$ (w.r.t.\ $\norm{\cdot}_{\hs\cap\hsmid}$).
Note that although $\hs\cap\hsmid$ may depend on $Z$, $D_{+}$ is independent of $Z$.
We will denote the extension of $\tilde{\iota}_{+}$ to $D_{+}$ by $\iota_{+}$, which can be expressed by $\iota_{+} = \cl{\tilde{\iota}_{+}}$. The adjoint $\iota_{+}\adjun$ coincides with $\tilde{\iota}_{+}\adjun$. Also $D_{-}$ does not change, if we replace $\tilde{D}_{+}$ by $D_{+}$ in \Cref{def:Dminus} and all previous results in this section also hold for $D_{+}$ and $\iota_{+}$ instead of $\tilde{D}_{+}$ and $\tilde{\iota}_{+}$, respectively.
If $\tilde{\iota}_{+}$ is already closed, then $D_{+} = \tilde{D}_{+}$.

\begin{lemma}\label{le:D-plus-characterization}
  Let one assertion in \Cref{prop:iota-closed-equivalences} be satisfied.
  Let $Z = Y\dual$, where $Y = D_{-}$ endowed with
  $\norm{g}_{Y}\coloneqq \norm{g}_{\hsdual\cap\hsmid} = \sqrt{\norm{g}_{\hsdual}^{2} + \norm{g}_{\hsmid}^{2}}$
  (from \Cref{prop:iota-closed-equivalences}
  $\ref{item:iota-closed-equivalences-iv}\Rightarrow \ref{item:iota-closed-equivalences-i}$).
  Then we have the following characterization for $D_{+}$:
  \begin{itemize}[itemsep = 4pt, topsep = 4pt]
    \item $D_{+} = \dom \iota_{-}\adjun$,

    \item $D_{+} = \hs\cap\hsmid$ in $Y\dual$.
  \end{itemize}
\end{lemma}

\begin{proof}
  Note that for $g \in D_{-}$ we have $g = (\iota_{+}\adjun)^{-1}\iota_{+}\adjun g$ and that $\iota_{+}\adjun\iota_{-}$ is isometric from $D_{-} = \dom \iota_{-} \subseteq \hsdual$ onto $\ran \iota_{+}\adjun = \dom (\iota_{+}\adjun)^{-1}\subseteq \hs\dual$. The following equivalences show the first assertion:
  \begin{align*}
    f \in \dom \iota_{-}\adjun
    &\Leftrightarrow D_{-} \ni g \mapsto \scprod{f}{\iota_{-} g}_{\hsmid}
      \text{ is continuous w.r.t.\ } \norm{\cdot}_{\hsdual}
    \\
    &\Leftrightarrow D_{-} \ni g \mapsto \scprod{f}{(\iota_{+}\adjun)^{-1}\iota_{+}\adjun \iota_{-} g}_{\hsmid}
      \text{ is continuous w.r.t.\ } \norm{\cdot}_{\hsdual}
    \\
    &\Leftrightarrow  \dom (\iota_{+}\adjun)^{-1} \ni h \mapsto \scprod{f}{(\iota_{+}\adjun)^{-1}h}_{\hsmid}
      \text{ is continuous w.r.t.\ } \norm{\cdot}_{\hs\dual}
    \\
    &\Leftrightarrow f \in \dom \big((\iota_{+}\adjun)^{-1}\big)\adjun = \dom \iota_{+}^{-1} = \ran \iota_{+} = D_{+}.
  \end{align*}

  For the second characterization we define $P_{+} \coloneqq \hs \cap \hsmid$ and we define $P_{-}$ analogously to $D_{-}$ in \Cref{def:Dminus}:
  \begin{equation*}
    \norm{g}_{P_{-}} \coloneqq \sup_{f\in P_{+}\setminus \sset{0}}\frac{\abs{\scprod{g}{f}_{\hsmid}}}{\norm{f}_{\hs}}
    \quad\text{and}\quad
    P_{-} \coloneqq \dset{g \in \hsmid}{\norm{g}_{P_{-}} < +\infty}.
  \end{equation*}
  Clearly, $\norm{g}_{\hsdual} \leq \norm{g}_{P_{-}}$ for $g \in P_{-}$ and consequently $P_{-} \subseteq D_{-}$.
  Furthermore, we can define $\iota_{P_{+}}\colon P_{+}\subseteq \hs \to \hsmid, f \mapsto f$ analogously to $\tilde{\iota}_{+}$.
  Note that $\iota_{P_{+}}$ is closed due the completeness of $(\hs\cap\hsmid,\norm{\cdot}_{\hs\cap\hsmid})$.
  Then we have $\dom \iota_{P_{+}}\adjun = P_{-}$ and $\tilde{\iota}_{+} \subseteq \iota_{P_{+}}$ and therefore $\iota_{P_{+}}\adjun \subseteq \tilde{\iota}_{+}\adjun$.
  For $g\in D_{-}$ and $f \in P_{+}$ we have, by definition of $P_{+} = \hs\cap\hsmid$ in $Z$,
  \begin{equation*}
    \abs{\scprod{g}{f}_{\hsmid}}
    = \abs{\scprod{\tilde{\iota}_{+}\adjun g}{f}_{\hs\dual,\hs}}
    \leq \norm{\tilde{\iota}_{+}\adjun g}_{\hs\dual} \norm{f}_{\hs}
    = \norm{g}_{\hsdual}\norm{f}_{\hs},
  \end{equation*}
  which yields $\norm{g}_{P_{-}} \leq \norm{g}_{\hsdual}$. Hence, $P_{-} = D_{-}$,
  $\iota_{P_{+}}\adjun = \tilde{\iota}_{+}\adjun$ and $\iota_{P_{+}} = \cl{\tilde{\iota}_{+}}$, which is equivalent to $P_{+}= \hs\cap\hsmid = \cl[\hs\cap\hsmid]{\tilde{D}_{+}} = D_{+}$.
\end{proof}

\begin{theorem}\label{th:hs-hsdual-dual-pairing}
  Let one assertion in \Cref{prop:iota-closed-equivalences} be satisfied. Then the continuous extension of $\iota_{+}\adjun \iota_{-}$ denoted by $\cl{\iota_{+}\adjun \iota_{-}}$ equals $\Psi$. Moreover, $\Psi$ is surjective and $(\hs,\hsdual)$ is a complete dual pair with
  \begin{equation*}
    \dualprod{g}{f}_{\hsdual,\hs} \coloneqq \dualprod{\Psi g}{f}_{\hs\dual,\hs}.
  \end{equation*}
\end{theorem}

\begin{proof}
  We have already shown, that $\iota_{+}\adjun \iota_{-} g = \Psi g$ for $g \in D_{-}$.
  Since $D_{-}$ is dense in $\hsdual$, we also have $\cl{\iota_{+}\adjun \iota_{-}} g = \Psi g$ for $g \in \hsdual$.

  If one assertion in \Cref{prop:iota-closed-equivalences} is true, then all of them are true.
  Hence, $\Psi D_{-}$ is dense in $\hs\dual$ and because $\Psi$ is isometric $\ran \Psi$ is closed and therefore $\ran \Psi = \hs\dual$.

  Since $\Psi$ is an isomorphism between $\hsdual$ and $\hs\dual$, it immediately follows that $(\hs,\hsdual)$ is a complete dual pair with the dual pairing $\dualprod{\cdot}{\cdot}_{\hsdual,\hs}$.
\end{proof}

\begin{remark}
  For $f \in D_{+}$ and $g\in D_{-}$ we have
  \begin{align*}
    \dualprod{g}{f}_{\hsdual,\hs} = \dualprod{\Psi g}{f}_{\hs\dual,\hs}
    = \dualprod{\iota_{+}\adjun \iota_{-}g}{f}_{\hs\dual,\hs}
    = \scprod{\iota_{-}g}{\iota_{+}f}_{\hsmid} = \scprod{g}{f}_{\hsmid}.
  \end{align*}
  Since these two sets are dense in $\hs$ and $\hsdual$ respectively, we have for $f\in \hs$ and $g\in\hsdual$
  \begin{align*}
    \dualprod{g}{f}_{\hsdual,\hs}
    =
    \lim_{(n,m) \to (\infty,\infty)}
    \scprod{g_{n}}{f_{m}}_{\hsmid},
  \end{align*}
  where $(f_{m})_{m\in\N}$ is a sequence in $D_{+}$ that converges to $f$ in
  $\hs$ and $(g_{n})_{n\in\N}$ is a sequence in $D_{-}$ that converges to $g$ in
  $\hsdual$.
\end{remark}

\section{Definition and Results}

The previous section leads to the following definition.

\begin{definition}\label{def:quasi-gelfand-triple}
  Let $(\hs,\hsdual)$ be a complete dual pair and $\hsmid$ be a Hilbert space.
  Furthermore, let $\iota_{+}\colon \dom \iota_{+} \subseteq \hs \to \hsmid$ and $\iota_{-}\colon \dom \iota_{-}\subseteq \hsdual \to \hsmid$ be densely defined, closed, and injective linear mappings with dense range.
  We call $(\hs,\hsmid,\hsdual)$ a \emph{pre-quasi Gelfand triple}, if
  \begin{equation}\label{eq:quasi-gelfand-triple-condition}
    \dualprod{g}{f}_{\hsdual,\hs} = \scprod{\iota_{-} g}{\iota_{+} f}_{\hsmid}
  \end{equation}
  for all $f \in \dom \iota_{+}$ and $g\in\dom\iota_{-}$. The space $\hsmid$ will be referred as \emph{pivot space}.

  If we additionally have $\dom \iota_{+}\adjun = \ran \iota_{-}$, then we call $(\hs,\hsmid,\hsdual)$ a \emph{quasi Gelfand triple}.
\end{definition}

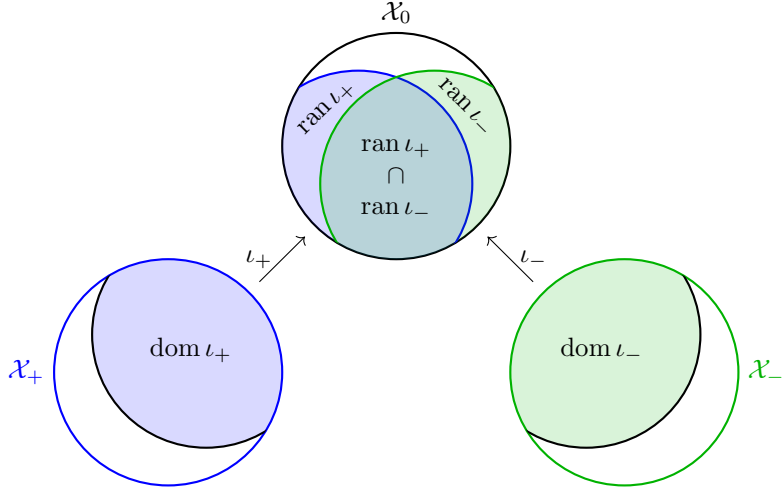
\begin{figure}
  \centering
  \begin{tikzpicture}
    \coordinate (v) at (0.5,0.5);
    \coordinate (w) at (-0.5,0.5);

    \coordinate (p0) at (0,0);
    \coordinate (p1) at (-3,-3);
    \coordinate (p2) at (3,-3);

    \draw[thick] (p0) circle [radius=1.5];
    \node[above] at ($(p0) + (0,1.5)$) {$\hsmid$};
    \begin{scope}
      \clip (p0) circle (1.5);
      \fill[blue,opacity=0.15] ($(p0)-(v)$) circle (1.5);
      \draw[blue,thick] ($(p0)-(v)$) circle (1.5);
    \end{scope}
    \begin{scope}
      \clip (p0) circle (1.5);
      \fill[black!30!green,opacity=0.15] ($(p0)-(w)$) circle (1.5);
      \draw[black!30!green,thick] ($(p0)-(w)$) circle (1.5);
    \end{scope}

    \node at ($(p0) + (-0.9,0.5)$) [rotate=45] {$\ran \iota_{+}$};
    \node at ($(p0) + (0.9,0.5)$) [rotate=-45] {$\ran \iota_{-}$};
    \node at ($(p0) + (0,-0.4)$) {\begin{tabular}{c}$\ran\iota_{+}$ \\ $\cap$ \\ $\ran \iota_{-}$\end{tabular}};

    \draw[blue,thick] (p1) circle [radius=1.5];
    \node[left,blue] at ($(p1) + (-1.5,0)$) {$\hs$};
    \begin{scope}
      \clip (p1) circle (1.5);
      \fill[blue,opacity=0.15] ($(p1)+(v)$) circle (1.5);
      \draw[thick] ($(p1)+(v)$) circle (1.5);
    \end{scope}
    \node at ($(p1)+(0.3,0.3)$) {$\dom \iota_{+}$};

    \draw[->] ($(p1) + 1.2*(1,1)$) -- node [left] {$\iota_{+}$} ($(p0) - 1.2*(1,1)$);

    \draw[thick,black!30!green] (p2) circle [radius=1.5];
    \node[right,black!30!green] at ($(p2) + (1.5,0)$) {$\hsdual$};
    \begin{scope}
      \clip (p2) circle (1.5);
      \fill[black!30!green,opacity=0.15] ($(p2)+(w)$) circle (1.5);
      \draw[thick] ($(p2)+(w)$) circle (1.5);
    \end{scope}
    \node at ($(p2)+(-0.3,0.3)$) {$\dom \iota_{-}$};

    \draw[->] ($(p2) + 1.2*(-1,1)$) -- node [right] {$\iota_{-}$}($(p0) + 1.2*(1,-1)$);

  \end{tikzpicture}
  \caption{Illustration of a quasi Gelfand triple}%
  \label{fig:quasi-gelfand-triple-abstract}
\end{figure}

\Cref{fig:quasi-gelfand-triple-abstract} illustrates the setting of a quasi Gelfand triple.
Contrary to the previous section we will regard the adjoint of $\iota_{+}$ and $\iota_{-}$ with respect to the complete dual pairs $(\hs,\hsdual)$ and $(\hsmid,\hsmid)$. Therefore, $\iota_{+}\adjun$ is a densely defined operator from $\hsmid$ to $\hsdual$ and $\iota_{-}\adjun$ is a densely defined operator from $\hsmid$ to $\hs$. We could not do this before, because we did not know from the beginning that $(\hs,\hsdual)$ is a complete dual pair.

\begin{example}
  Let $\hs = \Lp{p}(\R)$, $\hsdual = \Lp{q}(\R)$ and $\hsmid = \Lp{2}(\R)$, where $p\in (1,+\infty)$ and $\frac{1}{p} + \frac{1}{q} = 1$. Then $(\hs,\hsdual)$ is a complete dual pair. Note that $\Lp{p}(\R) \cap \Lp{2}(\R)$ is already well-defined. We can define
  \begin{align*}
    \iota_{+}&\colon
    \mapping{\Lp{p}(\R) \cap \Lp{2}(\R) \subseteq \Lp{p}(\R)}{\Lp{2}(\R)}{f}{f,}
    \\\mathllap{\text{and}\quad}
    \iota_{-}&\colon
    \mapping{\Lp{q}(\R) \cap \Lp{2}(\R) \subseteq \Lp{q}(\R)}{\Lp{2}(\R)}{g}{g.}
  \end{align*}
  These mapping are densely defined, injective and closed with dense range. By definition of the dual pairing of $(\Lp{p}(\R),\Lp{q}(\R))$ we have
  \begin{equation*}
    \dualprod{g}{f}_{\Lp{q}(\R),\Lp{p}(\R)} = \int_{\R} g\conj{f} \dx[\uplambda] = \scprod{g}{f}_{\hsmid}
    = \scprod{\iota_{-}g}{\iota_{+}f}_{\hsmid}
  \end{equation*}
  for $g \in \Lp{q}(\R) \cap \Lp{2}(\R)$ and $f \in \Lp{p}(\R) \cap \Lp{2}(\R)$. By the H{\"o}lder inequality we also have $\dom \iota_{+}\adjun = \ran \iota_{-}$. Hence, $(\Lp{p}(\R),\Lp{2}(\R),\Lp{q}(\R))$ is a quasi Gelfand triple.
\end{example}

Note that the mapping $\iota_{+}$ gives us an identification of $\dom \iota_{+}$ and $\ran \iota_{+}$. Hence, we can introduce the norm of $\hs$ on $\ran \iota_{+}$ by
$\norm{f}_{\hs} = \norm{\iota_{+}^{-1} f}_{\hs}$ for $f \in \ran \iota_{+}$.
Then the completion of $\ran \iota_{+}$ with respect to $\norm{\cdot}_{\hs}$ is isometrically isomorphic to $\hs$. Accordingly, we can do the same for $\hsdual$. This justifies the following definition and \Cref{fig:quasi-gelfand-triple}.

\begin{definition}\label{def:intersections-of-members-of-quasi-Gelfand-triple}
  For a quasi Gelfand triple $(\hs,\hsmid,\hsdual)$ we define
  \begin{equation*}
    \hs \cap \hsmid \coloneqq \ran \iota_{+}
    \quad\text{and}\quad
    \hsdual \cap \hsmid \coloneqq \ran \iota_{-}.
  \end{equation*}
\end{definition}

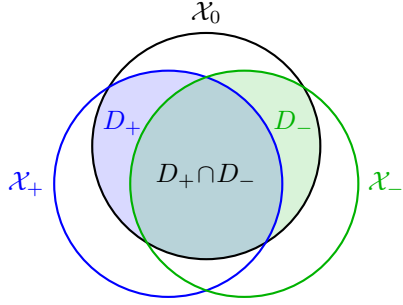
\begin{figure}
  \centering
  \begin{tikzpicture}
    \begin{scope}
      \clip (-0.5,-0.5) circle (1.5);
      \fill[blue,opacity=0.15] (0,0) circle (1.5);
    \end{scope}
    \begin{scope}
      \clip (0.5,-0.5) circle (1.5);
      \fill[black!30!green,opacity=0.15] (0,0) circle (1.5);
    \end{scope}
    \draw[thick](0,0) circle [radius=1.5];
    \draw[thick,blue](-0.5,-0.5) circle [radius=1.5];
    \draw[thick,black!30!green](0.5,-0.5) circle [radius=1.5];
    \node[left,blue] at (-2,-0.5) {$\hs$};
    \node[right,black!30!green] at (2,-0.5) {$\hsdual$};
    \node[above] at (0,1.5) {$\hsmid$};
    \node[blue] at (-1.1,0.3) {${D}_{+}$};
    \node[black!30!green] at (1.15,0.3) {$D_{-}$};
    \node at (0,-0.4) {${D}_{+}\!\cap\! D_{-}$};
  \end{tikzpicture}
  \caption{Illustration of a quasi Gelfand triple, where $D_{+} = \ran \iota_{+}$
    and $D_{-} = \ran \iota_{-}$.}%
  \label{fig:quasi-gelfand-triple}
\end{figure}

If either $\iota_{+}$ or $\iota_{-}$ is continuous, then a quasi Gelfand triple is an ``ordinary'' Gelfand triple.
Clearly, every ``ordinary'' Gelfand triple is also a quasi Gelfand triple.

The additional condition $\dom \iota_{+}\adjun = \ran \iota_{-}$ that makes a pre-quasi Gelfand triple a quasi Gelfand triple is not crucial as it can always be forced, which we will see later in \Cref{le:extension-of-iota-minus-to-fit-def}.
In \Cref{con:weak,con:strong} we ask ourselves, if this condition is automatically fulfilled.
Moreover, the next lemma shows that we can also ask for the converse condition $\dom \iota_{-}\adjun = \ran \iota_{+}$ instead.
Note that from~\eqref{eq:quasi-gelfand-triple-condition} we can immediately see that $\dom \iota_{+}\adjun \supseteq \ran \iota_{-}$ and $\dom \iota_{-}\adjun \supseteq \ran \iota_{+}$. Hence, for $f \in \dom \iota_{+}$ and $g \in \dom \iota_{-}$ we have
\begin{equation}\label{eq:iota-plus-adjun-equals-iota-minus-inverse}
  \dualprod{g}{f}_{\hsdual,\hs} = \scprod{\iota_{-} g}{\iota_{+} f}_{\hsmid}
  =
  \begin{cases}
    \dualprod{\iota_{+}\adjun \iota_{-} g}{f}_{\hsdual,\hs}, \\
    \dualprod{g}{\iota_{-}\adjun \iota_{+} f}_{\hsdual,\hs},
  \end{cases}
\end{equation}
which implies $\iota_{+}\adjun \iota_{-} g = g$ and $\iota_{-}\adjun \iota_{+} f = f$.

\begin{lemma}\label{le:iota-adjun-dom}
  Let $(\hs,\hsmid,\hsdual)$ be a pre-quasi Gelfand triple with the embeddings $\iota_{+}$ and $\iota_{-}$.
  Then
  \begin{align*}
    \dom \iota_{+}\adjun = \ran \iota_{-}
    \quad\Leftrightarrow\quad
    \dom \iota_{-}\adjun = \ran \iota_{+}.
  \end{align*}
  In particular, if $(\hs,\hsmid,\hsdual)$ is a quasi Gelfand triple, then also $\dom \iota_{-}\adjun = \ran \iota_{+}$ holds true.
\end{lemma}

The proof of this is basically the first part of the proof of \Cref{le:D-plus-characterization}.

\begin{proof}

  Let $\dom \iota_{+}\adjun  = \ran \iota_{-}$.
  The following equivalences
  \begin{align*}
    f \in \dom \iota_{-}\adjun
    &\Leftrightarrow
      \dom\iota_{-} \ni g \mapsto \scprod{f}{\iota_{-} g}_{\hsmid}
      \text{ is continuous w.r.t.\ } \norm{\cdot}_{\hsdual}
    \\
    &\Leftrightarrow
      \dom\iota_{-} \ni g \mapsto \scprod{f}{(\iota_{+}\adjun)^{-1}\underbrace{\iota_{+}\adjun \iota_{-} g}_{=\mathrlap{g}}}_{\hsmid}
      \text{ is continuous w.r.t.\ } \norm{\cdot}_{\hsdual}
      \\
    &\Leftrightarrow
      f \in \dom \big((\iota_{+}\adjun)^{-1}\big)\adjun = \dom \iota_{+}^{-1} = \ran \iota_{+}
  \end{align*}
  imply $\dom \iota_{-}\adjun = \ran \iota_{+}$.

  The other implication follows  analogously.
\end{proof}

In contrast to ``ordinary'' Gelfand triple, the setting for quasi Gelfand triple is somehow ``symmetric'', i.e., the roles of $\hs$ and $\hsdual$ are interchangeable, since neither of the embeddings $\iota_{+}$ and $\iota_{-}$ has to be continuous, as indicated in the beginning of this section.

\begin{lemma}
  \label{le:extension-of-iota-minus-to-fit-def}
  Let $(\hs,\hsmid,\hsdual)$ be a pre-quasi Gelfand triple with the embeddings $\iota_{+}$ and $\iota_{-}$.
  Then there exists an extension $\hat{\iota}_{-}$ of $\iota_{-}$ that respects~\eqref{eq:quasi-gelfand-triple-condition} and satisfies $\dom \iota_{+}\adjun = \ran \hat{\iota}_{-}$.
  In particular, $(\hs,\hsmid,\hsdual)$ with $\iota_{+}$ and $\hat{\iota}_{-}$ forms a quasi Gelfand triple.
\end{lemma}

\begin{proof}
  Note that $\iota_{+}\adjun \iota_{-} g = g$. Hence, $\iota_{+}\adjun \supseteq \iota_{-}^{-1}$ and $(\iota_{+}\adjun)^{-1} \supseteq \iota_{-}$. We define $\hat{\iota}_{-}$ as $(\iota_{+}\adjun)^{-1}$.
  Then clearly $\ran \hat{\iota}_{-} = \dom \iota_{+}\adjun$.
  For $f \in \dom \iota_{+}$ and $g \in \dom \hat{\iota}_{-}$ we have
  \begin{equation*}
    \scprod{\hat{\iota}_{-} g}{\iota_{+}f}_{\hsmid} = \dualprod{\iota_{+}\adjun \tilde{\iota}_{-}g}{f}_{\hsdual,\hs} = \dualprod{g}{f}_{\hsdual,\hs}.
    \ifSn{\tag*{\qedhere}}{\qedhere}
  \end{equation*}
\end{proof}

Alternatively, we could have extended $\iota_{+}$ by setting $\hat{\iota}_{+} \coloneqq (\iota_{-}\adjun)^{-1}$ in the previous lemma to obtain a quasi Gelfand triple.

\begin{remark}\label{re:quasi-gelfand-triple-different-dual-pair}
  If $(\hs,\hsmid,\hsdual)$ is a quasi Gelfand triple and $(\hs,\widetilde{\hsdual})$ is another dual pair for $\hs$, then also $(\hs,\hsmid,\widetilde{\hsdual})$ is a quasi Gelfand triple.
\end{remark}

\begin{lemma}\label{le:iota-plus-adjun=iota-minus-inverse}
  Let $(\hs,\hsmid,\hsdual)$ be a quasi Gelfand triple. Then
  \begin{equation*}
    \iota_{+}\adjun = \iota_{-}^{-1}
    \quad\text{and}\quad
    \iota_{-}\adjun = \iota_{+}^{-1}.
  \end{equation*}
\end{lemma}

\begin{proof}
  By~\eqref{eq:iota-plus-adjun-equals-iota-minus-inverse} we have $\iota_{+}\adjun \iota_{-} g = g$ for all $g \in \dom \iota_{+}$. Since $\ran \iota_{-} = \dom \iota_{+}\adjun$ (by assumption), we conclude that $\iota_{+}\adjun = \iota_{-}^{-1}$.

  Analogously, the second equality can be shown.
\end{proof}

\begin{theorem}\label{th:quasi-gelfand-triple-charaterization}
  Let $\hs$ be a reflexive Banach space and $\hsmid$ be a Hilbert space and $\iota_{+}\colon \dom \iota_{+} \subseteq \hs \to \hsmid$ be a densely defined, closed, and injective linear mapping with dense range. Then there exists a Banach space $\hsdual$ and a mapping $\iota_{-}$ such that $(\hs,\hsmid,\hsdual)$ is a quasi Gelfand triple.

  In particular, $\hsdual$ is given by \Cref{def:Dminus}, where $D_{+} = \ran \iota_{+}$.
\end{theorem}

\begin{proof}
  We will identify $\dom \iota_{+}$ with $\ran \iota_{+}$ and denote it by $D_{+}$.
  Then \cref{item:iota-closed-equivalences-iii} of \Cref{prop:iota-closed-equivalences} is satisfied.
  Hence, the corresponding $D_{-}$ (\Cref{def:Dminus}) is dense in $\hsmid$ and its completion
  $\hsdual$ (w.r.t.\ to $\norm{\cdot}_{\hsdual}$) establishes the complete dual pair $(\hs,\hsdual)$, by \Cref{th:hs-hsdual-dual-pairing}.
  The mapping
  \begin{equation*}
    \iota_{-} \colon
    \mapping{D_{-}\subseteq \hsdual}{\hsmid}{g}{g,}
  \end{equation*}
  is densely defined and injective by construction. By the already shown $\ran \iota_{-} = D_{-}$ is dense in $\hsmid$. Finally, by \Cref{th:iotap-and-iotam} $\iota_{-}$ is closed and by \Cref{le:adjoint-embedding} $\dom \iota_{+}\adjun = D_{-}= \ran \iota_{-}$.
\end{proof}

\begin{remark}\label{re:setting-yields-gelfand-triple}
  By \Cref{th:quasi-gelfand-triple-charaterization} the setting in the beginning of \Cref{sec:motivation} establishes a quasi Gelfand triple, if one assertion of \Cref{prop:iota-closed-equivalences} is satisfied.
\end{remark}


\begin{framed}
  \noindent
  From now on we will assume that $(\hs,\hsmid,\hsdual)$ is a quasi Gelfand triple and we will identify $\dom \iota_{+}$ with $\ran \iota_{+}$ and denote it by $D_{+}$ as in \Cref{fig:quasi-gelfand-triple}. Analogously, we identify $\dom \iota_{-}$ with $\ran \iota_{-}$ and denote it with $D_{-}$.
\end{framed}

These identifications are really meaningful as we can endow $D_{+}$ (as a subset of $\hsmid$) with $\norm{f}_{\hs} \coloneqq \norm{\iota_{+}^{-1}f}_{\hs}$ for $f \in D_{+}$. Then the completion of $D_{+}$ w.r.t.\ to this norm is clearly isomorphic to $\hs$. The same goes for $D_{-}$.

The set $D_{-} = \ran \iota_{-}$ (previous identification) coincides with the set $D_{-}$ defined in \Cref{def:Dminus} for $\tilde{D}_{+} \coloneqq D_{+}$.

\begin{proposition}\label{th:Dplus-cap-Dminus-complete-with-Xplus-cap-Xminus-norm}
  The space $D_{+}\cap D_{-}$ is complete with respect to
  \begin{equation*}
    \norm{\cdot}_{\hs\cap\hsdual} \coloneqq \sqrt{\norm{\cdot}_{\hs}^{2} + \norm{\cdot}_{\hsdual}^{2}}
    \quad\text{and}\quad
    \norm{f}_{\hsmid} \leq \norm{f}_{\hs\cap \hsdual} \quad \forall f\in D_{+}\cap D_{-}.
  \end{equation*}
\end{proposition}

\begin{proof}
  For $f \in D_{+}\cap D_{-}$ we have
  \begin{align*}
    \norm{f}_{\hsmid}^{2} = \abs{\scprod{f}{f}_{\hsmid}} = \abs{\dualprod{f}{f}_{\hsdual,\hs}} \leq \norm{f}_{\hsdual} \norm{f}_{\hs} \leq \norm{f}_{\hs\cap\hsdual}^{2}.
  \end{align*}
  Hence, every Cauchy sequence in $D_{+}\cap D_{-}$ with respect to $\norm{\cdot}_{\hs\cap\hsdual}$ is also a Cauchy sequence with respect to $\norm{\cdot}_{\hsmid}$, $\norm{\cdot}_{\hs}$ and $\norm{\cdot}_{\hsdual}$.

  Let $(f_{n})_{n\in\N}$ be a Cauchy sequence in $D_{+}\cap D_{-}$ with respect to $\norm{\cdot}_{\hs\cap\hsdual}$.
  By the closedness of $\iota_{+}$ the limit with respect to $\norm{\cdot}_{\hsmid}$ and the limit with respect to $\norm{\cdot}_{\hs}$ coincide.
  The same argument for $\iota_{-}$ yields that the limit with respect to $\norm{\cdot}_{\hsmid}$ and the limit with respect $\norm{\cdot}_{\hsdual}$ also coincide.
  Therefore, all these limits have to coincide and $(f_{n})_{n\in\N}$ converges to that limit in $D_{+}\cap D_{-}$ w.r.t.\ $\norm{\cdot}_{\hs\cap\hsdual}$.
\end{proof}

\begin{lemma}\label{th:iota-p-iota-m-closed}
  The operator
  \begin{align*}
    \begin{bmatrix}
      \iota_{+}    & \iota_{-}
    \end{bmatrix}
        \colon
        \mapping{D_{+} \times D_{-} \subseteq \hs \times \hsdual}{\hsmid}{\begin{bmatrix}
          f \\ g
        \end{bmatrix}}{f+g,}
  \end{align*}
  is closed.
\end{lemma}

\begin{proof}
  Let $\left(\left(\begin{bsmallmatrix}f_n \\ g_n\end{bsmallmatrix},z_n\right)\right)_{n\in\N}$ be a sequence in
  \(
  \begin{bmatrix}
    \iota_{+} & \iota_{-}
  \end{bmatrix}
  \)
  that converges to
  \(
  \left(
    \begin{bsmallmatrix}f \\ g\end{bsmallmatrix},z
  \right)
  \)
  in $\hs\times \hsdual \times \hsmid$, i.e.,
  \begin{align*}
    \lim_{n\to\infty} f_{n} &= f\quad (\text{w.r.t.}\ \norm{\cdot}_{\hs}), \\
    \lim_{n\to\infty} g_{n} &= g\quad (\text{w.r.t.}\ \norm{\cdot}_{\hsdual}), \\
    \mathllap{\text{and}\quad}\lim_{n\to\infty} f_{n} + g_{n} = \lim_{n\to\infty} z_{n} &= z\quad (\text{w.r.t.}\ \norm{\cdot}_{\hsmid}).
  \end{align*}
  Then we have
    \begin{align*}
      \norm{z}_{\hsmid}^2 = \lim_{n\to\infty} \norm{f_n + g_n}_{\hsmid}^2 = \lim_{n\to\infty}\big( \norm{f_n}_{\hsmid}^2 + \norm{g_n}_{\hsmid}^2 + 2\Re \scprod{f_n}{g_n}_{\hsmid}\big).
    \end{align*}
    Since $2\Re \scprod{f_n}{g_n}_{\hsmid}$ converges to
    $2\Re\dualprod{f}{g}_{\hs,\hsdual}$, we conclude that $\norm{f_n}_{\hsmid}$
    and $\norm{g_n}_{\hsmid}$ are bounded. Hence, there exists a subsequence
    of $(f_{n})_{n\in\N}$ that converges weakly (in $\hsmid$) to an $\tilde{f}\in\hsmid$.
    Moreover, by \Cref{th:weak-to-strong-convergent} we can pass on to a further
    subsequence  $(f_{n(k)})_{k\in\N}$ such that $\big(\frac{1}{j}\sum_{k=1}^{j}f_{n(k)}\big)_{j\in\N}$ converges to $\tilde{f}$ strongly (w.r.t.\ $\norm{\cdot}_{\hsmid}$). The sequence $\big(\frac{1}{j}\sum_{k=1}^{j}f_{n(k)}\big)_{j\in\N}$ has still the limit $f$ in $\hs$ (w.r.t.\ $\norm{\cdot}_{\hs}$) and because $\iota_{+}$ is closed we conclude that $f=\tilde{f}\in D_{+}$.
    By linearity of the limit we also have $\frac{1}{j}\sum_{k=1}^{j}g_{n(k)} \to z - f$ in $\hsmid$ for the same subsequence.
    Since $\frac{1}{j}\sum_{k=1}^{j}g_{n(k)}$ is a Cauchy sequence in both $\hsdual$ and $\hsmid$, the closedness of $\iota_{-}$ gives that $g = z-f \in D_{-}$.
    Hence,
    $z=
    \begin{bmatrix}
      \iota_{+} & \iota_{-}
    \end{bmatrix}
    \left[
      \begin{smallmatrix}
        f\\g
      \end{smallmatrix}
    \right]$
    and the operator $\begin{bmatrix}\iota_{+} & \iota_{-}\end{bmatrix}$ is closed.
\end{proof}

\begin{proposition}\label{th:Dplus-cap-Dminus-dense-in-X0}
  $D_{+} \cap D_{-}$ is dense in $\hsmid$ with respect to $\norm{\cdot}_{\hsmid}$.
\end{proposition}

\begin{proof}
  By $\dom \iota_{\pm}\adjun = \ran \iota_{\mp} = D_{\mp}$ (\Cref{le:iota-adjun-dom}) and
  \(
  \mul
  \begin{bmatrix}
    \iota_{+} & \iota_{-}
  \end{bmatrix}
  = \sset{0}
  \)
  we have
  \begin{equation*}
    \hsmid =
    \big(
    \mul
    \begin{bmatrix}
      \iota_{+} & \iota_{-}
    \end{bmatrix}
    \big)^{\perp}
    =
    \cl{%
      \dom
      \begin{bmatrix}
        \iota_{+} & \iota_{-}
      \end{bmatrix}\adjun
    }
    = \cl{\dom \iota_{+}\adjun \cap \dom \iota_{-}\adjun}
    = \cl{D_{-} \cap D_{+}}.
    \ifSn{\tag*{\qedhere}}{\qedhere}
  \end{equation*}
\end{proof}

\section{Quasi Gelfand Triples with Hilbert Spaces}
\label{sec:quasi-Gelfand-triples-with-Hilbert-spaces}

In this section we will regard a quasi Gelfand triple $(\hs,\hsmid,\hsdual)$, where $\hs$ and $\hsdual$ (and of course $\hsmid$) are Hilbert spaces. Maybe also some these results can be proven for general quasi Gelfand triple, but we would need a replacement for \Cref{th:TTadjun-self-adjoint}.

For a quasi Gelfand triple $(\hs,\hsmid,\hsdual)$ consisting of Hilbert spaces, there exists a unitary mapping $\Psi$ from $\hsdual$ to $\hs$ (Riesz representation theorem) satisfying
\begin{equation*}
  \dualprod{g}{f}_{\hsdual,\hs} = \scprod{\Psi g}{f}_{\hs}
  \quad\text{and}\quad
  \dualprod{f}{g}_{\hs,\hsdual} = \scprod{\Psi^{-1} f}{g}_{\hsdual}.
\end{equation*}
We will refer to this mapping $\Psi$ as the \emph{duality map} of the quasi Gelfand triple.

Note that we previously regarded the adjoint of $\iota_{+}$ with respect to the dual pairs $(\hsmid,\hsmid)$ and $(\hs,\hsdual)$.
The main reason for this choice was, that if $\hs$ is not a Hilbert space, then the dual pair $(\hs,\hs)$ is not available, but also sometimes the adjoint with respect to the dual pair $(\hs,\hsdual)$ is more natural.

However, now that $\hs$ is a Hilbert space, the dual pairs $(\hs,\hs)$ and $(\hsdual,\hsdual)$ are available and seem reasonable when it comes to calculating adjoints. Hence, if we have an additional dual pair $(Y,Z)$ and a linear operator $A$ from $\hs$ to $Y$, then we have two choices for the adjoint:
\begin{equation*}
  A\adjunX{Z \times \hs}\colon \dom A\adjun \subseteq Z \to \hs \quad\text{and}\quad A\adjunX{Z \times \hsdual}\colon \dom A\adjun \subseteq Z \to \hsdual,
\end{equation*}
as defined in \Cref{def:adjoint-relation}. In order to have a short notation we will denote the adjoints that are taken w.r.t.\ the dual pairs $(\hs,\hs)$ and $(\hsdual,\hsdual)$ by $A\hadjun$ ($\mathrm{h}$ for Hilbert space duality) and the adjoints w.r.t.\ $(\hs,\hsdual)$ still by $A\adjun$, i.e.,
\begin{equation*}
  A\hadjun \colon \dom A\hadjun \subseteq Z \to \hs \quad\text{and}\quad A\adjun \colon \dom A\adjun \subseteq Z \to \hsdual.
\end{equation*}
Clearly, the same goes for mappings, where $\hs$ is the codomain and analogously for $\hsdual$.
Note that for $\hsmid$ we regard only the dual pair $(\hsmid,\hsmid)$, therefore we always take adjoints with respect to this dual pair. In particular for $\iota_{+}$ we have
\begin{equation*}
  \iota_{+}\hadjun \colon \dom \iota_{+}\hadjun \subseteq \hsmid \to \hs \quad\text{and}\quad \iota_{+}\adjun \colon \dom \iota_{+}\adjun \subseteq \hsmid \to \hsdual.
\end{equation*}
By \Cref{le:adjoints-for-different-dualities} we have the following relations between the adjoints:
\begin{equation*}
  \iota_{+}\hadjun = \Psi \iota_{+}\adjun
  \quad\text{and}\quad
  \iota_{-}\hadjun = \Psi^{-1} \iota_{-}\adjun.
\end{equation*}

\begin{corollary}\label{th:Dplus-Dminus-core-of-iota-adjun-iota}
  The set $D_{+}\cap D_{-}$ is dense in $\hs$ and $\hsdual$ with respect to their corresponding norms. More precisely $\dom \iota_{+}\adjun \iota_{+} = \iota_{+}^{-1}(D_{+} \cap D_{-})$ is dense in $\hs$ and $\dom \iota_{-}\adjun \iota_{-} = \iota_{-}^{-1}(D_{+} \cap D_{-})$ is dense in $\hsdual$.

  Furthermore, $\iota_{+}^{-1}(D_{+} \cap D_{-})$ is a core of $\iota_{+}$ and $\iota_{-}^{-1}(D_{+} \cap D_{-})$ is a core of $\iota_{-}$.
\end{corollary}

\begin{proof}
  Applying \Cref{th:TTadjun-self-adjoint} to $\iota_{+}$ yields $\iota_{+}\hadjun \iota_{+}$ is self-adjoint.
  Note that by \Cref{le:adjoints-for-different-dualities} we have $\iota_{+}\hadjun = \Psi \iota_{+}\adjun$, where $\Psi$ is the duality map introduced in the beginning of this section.
  Hence, $\dom \iota_{+}\hadjun \iota_{+} = \dom \iota_{+}\adjun \iota_{+}$ is dense in $\hs$.
  By \Cref{le:iota-adjun-dom} $\dom \iota_{+}\adjun = D_{-}$, consequently
  \begin{equation}\label{eq:dom-iota-adjun-iota-equals-Dp-cap-Dm}
    \dom \iota_{+}\adjun \iota_{+} = \iota_{+}^{-1}(\dom \iota_{+}\adjun \cap \ran \iota_{+}) = \iota_{+}^{-1}(D_{-} \cap D_{+}) = D_{+}\cap D_{-}.
  \end{equation}
  Finally, \Cref{th:TadjunT-selfadjoint-for-dual-pair} and \eqref{eq:dom-iota-adjun-iota-equals-Dp-cap-Dm} gives that $\iota_{+}^{-1}(D_{+} \cap D_{-})$ is a core of $\iota_{+}$.

  An analogous argument for $\iota_{-}$ yields $D_{+}\cap D_{-}$ is dense in $\hsdual$.
\end{proof}

\begin{corollary}\label{co:Dplus-plus-Dminus-pivot-space}
  $D_{+} + D_{-} = \hsmid$.
\end{corollary}

\begin{proof}
  Applying \Cref{th:TTadjun-self-adjoint} to $\iota_{+}$ gives that $(\opid_{\hsmid} + \iota_{+}\iota_{+}\hadjun)$ is onto.
  Hence, for every $x\in \hsmid$ there exists a $g_{x} \in \dom \iota_{+}\iota_{+}\hadjun \subseteq D_{-}$ such that
  \begin{equation*}
    x = \underbrace{g_x}_{\in \mathrlap{D_{-}}} + \underbrace{\iota_{+} \iota_{+}\hadjun g_x}_{\in \mathrlap{D_{+}}}.
  \end{equation*}
  Since  $g_{x} \in \dom \iota_{+}\iota_{+}\hadjun$, we have $\iota_{+}\hadjun g_x \in D_{+}$ and consequently $x \in D_{+} + D_{-}$.
\end{proof}

Next we will show that we can embed an entire quasi Gelfand triple structure preservingly into a larger space. We will even give the smallest possible space that contains the entire quasi Gelfand triple. However, before we start we give a proper definition of what we mean.

\begin{definition}\label{def:structure-preservingly-embedded}
  Let $\mathcal{H}$ be a Hausdorff topological vector space. We say the quasi Gelfand triple $(\hs,\hsmid,\hsdual)$ can be \emph{structure preservingly embedded} into $\mathcal{H}$, if there exist linear, injective and continuous mappings
  \begin{align*}
    \phi_{\hs}\colon \hs \to \mathcal{H},
    \quad
    \phi_{\hsmid}\colon \hsmid \to \mathcal{H}
    \quad\text{and}\quad
    \phi_{\hsdual}\colon \hsdual \to \mathcal{H}
  \end{align*}
  such that
  \begin{align}\label{eq:structure-preserving-emdding-conditions}
    \phi_{\hs}\big\vert_{\dom \iota_{+}} = \phi_{\hsmid}\iota_{+} \quad\text{and}\quad \phi_{\hsdual}\big\vert_{\dom \iota_{-}} = \phi_{\hsmid} \iota_{-}.
  \end{align}
\end{definition}

Basically the previous definition means that the following diagram commutes.
\begin{equation*}
  \begin{tikzcd}[column sep=1em, row sep=1em]
    & & & \mathcal{H} & & &\\
    \\
    \hs\ar{rrruu}{\phi_{\hs}}&  & & \hsmid\ar{uu}[swap]{\phi_{\hsmid}} & & & \hsdual\ar{llluu}[swap]{\phi_{\hsdual}} \\
    &\dom \iota_{+}\ar{lu}{\id} \ar[shift left]{r}{\iota_{+}} & \ran \iota_{+}\ar[shift left]{l}{\iota_{+}^{-1}}\ar{ru}[swap]{\id} & & \ran \iota_{-}\ar{lu}{\id}\ar[shift right]{r}[swap]{\iota_{-}^{-1}} & \dom \iota_{-}\ar[shift right]{l}[swap]{\iota_{-}}\ar[swap]{ru}{\id}&
  \end{tikzcd}
\end{equation*}
Since we identify $\dom \iota_{+}$ and $\ran \iota_{+}$ with each other and denote it as $D_{+}$ and the same for $\iota_{-}$, we can reduce the previous diagram to the following diagram.
\begin{equation*}
  \begin{tikzcd}[column sep=1em, row sep=1em]
    & & \mathcal{H} & & \\
    \\
    \hs\ar{rruu}{\phi_{\hs}} & & \hsmid\ar{uu}[swap]{\phi_{\hsmid}} & & \hsdual\ar{lluu}[swap]{\phi_{\hsdual}} \\
    & D_{+}\ar{lu}{\id}\ar{ru}[swap]{\id} & & D_{-}\ar{lu}{\id}\ar{ru}[swap]{\id} &
  \end{tikzcd}
\end{equation*}
From this point of view the compatibility condition~\eqref{eq:structure-preserving-emdding-conditions} can be seen as
\begin{equation*}
  \phi_{\hs} f = \phi_{\hsmid} f \quad\forall f \in D_{+}
  \quad\text{and}\quad
  \phi_{\hsdual} g = \phi_{\hsmid} g \quad\forall g \in D_{-}.
\end{equation*}


Note if $(\hs,\hsmid,\hsdual)$ is an ``ordinary'' Gelfand triple (where $\iota_{+}$ is continuous), then it is usually denoted by $\hs \subseteq \hsmid \subseteq \hsdual$. To be precise these inclusions are actually identifications via the mappings $\iota_{+}$ and $\iota_{-}^{-1}$. The continuity and closedness of $\iota_{+}$ implies $\dom \iota_{+} = \hs$ and that $\iota_{+}\adjun$ is also continuous and everywhere defined.
Since $\iota_{+}\adjun = \iota_{-}^{-1}$ (\Cref{le:iota-plus-adjun=iota-minus-inverse}), we have the following setting:
\begin{equation*}
  \begin{tikzcd}
    \hs \ar{r}{\iota_{+}} & \hsmid \ar{r}{\iota_{-}^{-1}}& \hsdual,
  \end{tikzcd}
\end{equation*}
which suggests that $\hsdual$ contains the entire Gelfand triple.
Defining $\phi_{\hs} = \iota_{-}^{-1}\iota_{+}$, $\phi_{\hsmid} = \iota_{-}^{-1}$ and $\phi_{\hsdual} = \id_{\hsdual}$ justifies that $\hsdual$ contains the Gelfand triple in a structure preserving manner as defined in \Cref{def:structure-preservingly-embedded}.

For quasi Gelfand triple the construction of a space that covers the entire quasi Gelfand triple needs a bit more attention.

\smallskip

By \Cref{th:Dplus-cap-Dminus-complete-with-Xplus-cap-Xminus-norm}, $D_{+} \cap D_{-}$ with $\norm{\cdot}_{\hs\cap\hsdual}$ is complete and therefore a Banach space. Since $\hs$ and $\hsdual$ are Hilbert spaces (in this section) we can define the inner product
\begin{equation*}
  \scprod{g}{f}_{\hs\cap\hsdual} \coloneqq
  \scprod{g}{f}_{\hs} + \scprod{g}{f}_{\hsdual}
\end{equation*}
on $D_{+} \cap D_{-}$. This inner product induces the previous norm $\norm{\cdot}_{\hs\cap\hsdual}$. Consequently $D_{+} \cap D_{-}$ is a Hilbert space with $\scprod{\cdot}{\cdot}_{\hs\cap\hsdual}$.
For shorter notation we denote $D_{+} \cap D_{-}$ by $\Zp$, the corresponding inner product and norm by $\scprod{\cdot}{\cdot}_{\Zp}$ and $\norm{\cdot}_{\Zp}$, respectively.

\begin{corollary}
  Let $\Zp = D_{+} \cap D_{-}$ be the space defined in the previous paragraph.
  Then the triple $(\Zp,\hsmid,\Zp\dual)$ forms an ``ordinary'' Gelfand triple. In particular $\Zp\dual$ is isometrically isomorphic to $\Zm$, the completion of
  $\hsmid$ w.r.t.\
  \begin{equation*}
    \norm{h}_{\Zm} \coloneqq \sup_{z \in \Zp\setminus\sset{0}} \frac{\abs{\scprod{h}{z}_{\hsmid}}}{\norm{z}_{\Zp}}.
  \end{equation*}
\end{corollary}

\begin{proof}
  By \Cref{th:Dplus-cap-Dminus-dense-in-X0} we know that $\Zp$ is dense in $\hsmid$ and by \Cref{th:Dplus-cap-Dminus-complete-with-Xplus-cap-Xminus-norm} that the mapping $\iota_{\Zp}\colon \Zp \to \hsmid$, $z \mapsto z$ is continuous. Hence, ``ordinary'' Gelfand triple theory or \Cref{th:quasi-gelfand-triple-charaterization} gives the assertion.
\end{proof}


\begin{theorem}
  We can structure preservingly embed the quasi Gelfand triple $(\hs,\hsmid,\hsdual)$ into the space $\Zp\dual$ by the embeddings
  \begin{equation*}
    \psi_{\hs} f = \dualprod{f}{\iota_{-}^{-1} \cdot}_{\hs,\hsdual},\quad
    \psi_{\hsmid} h = \scprod{h}{\cdot}_{\hsmid}
    \quad\text{and}\quad
    \psi_{\hsdual} g = \dualprod{g}{\iota_{+}^{-1} \cdot}_{\hsdual,\hs}.
  \end{equation*}
\end{theorem}

Note that by our identifications of $D_{+}$ and $D_{-}$ we have $\iota_{+}^{-1} z = z$ and $\iota_{-}^{-1}z = z$ for $z \in \Zp$. However, making this change of spaces visible can sometimes help. Nevertheless, most of the time this is only additional dead weight, this is why we will often just write $\phi_{\hs}(f)(z) = \dualprod{f}{z}_{\hs,\hsdual}$, etc..

Clearly, since $\Zp\dual$ and $\Zm$ are isometrically isomorphic we can also structure preservingly embed $(\hs,\hsmid,\hsdual)$ into $\Zm$. For notational harmony we prefer to use $\Zm$ instead of $\Zp\dual$. However, for our purpose there is no need to strictly distinguish between them, this is why we will use these symbols as synonyms. \Cref{fig:embedded-gelfand-triple} illustrates the meaning of the previous theorem.

\begin{proof}
  First we have to check that these mappings are well-defined: Let $z \in \Zp$, $f \in \hs$, $h \in \hsmid$ and $g \in \hsdual$. Then
  \begin{align*}
    \abs{\psi_{\hs}(f)(z)} &= \abs{\dualprod{f}{z}_{\hs,\hsdual}} \leq \norm{f}_{\hs} \norm{z}_{\hsdual} \leq \norm{f}_{\hs} \norm{z}_{\Zp}, \\
    \abs{\psi_{\hsmid}(h)(z)} &= \abs{\scprod{h}{z}_{\hsmid}} \leq \norm{h}_{\hsmid} \norm{z}_{\hsmid} \leq \norm{h}_{\hsmid} \norm{z}_{\Zp}, \\
    \abs{\psi_{\hsdual}(g)(z)} &= \abs{\dualprod{g}{z}_{\hsdual,\hs}} \leq \norm{g}_{\hsdual} \norm{z}_{\hs} \leq \norm{g}_{\hsdual} \norm{z}_{\Zp},
  \end{align*}
  which implies $\psi_{\hs}(f)$, $\psi_{\hsmid}(h)$ and $\psi_{\hsdual}(g)$ are in $\Zp\dual$, and $\psi_{\hs}$, $\psi_{\hsmid}$ and $\psi_{\hsdual}$ are continuous. The linearity of $\psi_{\hs}$, $\psi_{\hsmid}$ and $\psi_{\hsdual}$ follows from the sesquilinearity of a dual pairing. If $\psi_{\hs}(f) = 0$, then $f \perp \iota_{-}^{-1} \Zp = \iota_{-}^{-1} (D_{+} \cap D_{-}) = \dom \iota_{-}\adjun \iota_{-}$. Since $\dom \iota_{-}\adjun \iota_{-}$ is dense in $\hsdual$, we conclude $f=0$, which proves $\phi_{\hs}$ is injective. Analogously, we can show that $\psi_{\hsdual}$ is injective. If $\psi_{\hsmid}(h) = 0$, then $h \perp \Zp$. Since $\Zp$ is dense in $\hsmid$, $h$ has to be $0$, which gives the injectivity of $\psi_{\hsmid}$.
  The compatibility condition~\eqref{eq:structure-preserving-emdding-conditions} follows from
  \begin{align*}
    \psi_{\hsmid}\circ\iota_{+}(f)(z)
    &= \scprod{\iota_{+}f}{z}_{\hsmid} = \dualprod{f}{\iota_{+}\adjun z}_{\hs,\hsdual}
      = \dualprod{f}{\iota_{-}^{-1} z}_{\hs,\hsdual} = \psi_{\hs}(f)(z),
    \\
    \psi_{\hsmid}\circ\iota_{-}(g)(z)
    &= \scprod{\iota_{-}g}{z}_{\hsmid} = \dualprod{g}{\iota_{-}\adjun z}_{\hsdual,\hs}
      = \dualprod{g}{\iota_{+}^{-1} z}_{\hsdual,\hs} = \psi_{\hsdual}(g)(z).
      \ifSn{\tag*{\qedhere}}{\qedhere}
  \end{align*}
\end{proof}

Now since we can always structure preservingly embed a quasi Gelfand triple into $\Zm$ ($\Zp\dual$) we can regard this quasi Gelfand triple as subsets of $\Zm$, see \Cref{fig:embedded-gelfand-triple-venn-diagram}, and do not have to deal with all this embeddings (most of the time). However, we will not get completely rid of these embeddings, as they are sometimes helpful, but we can always regard them as identity mappings.

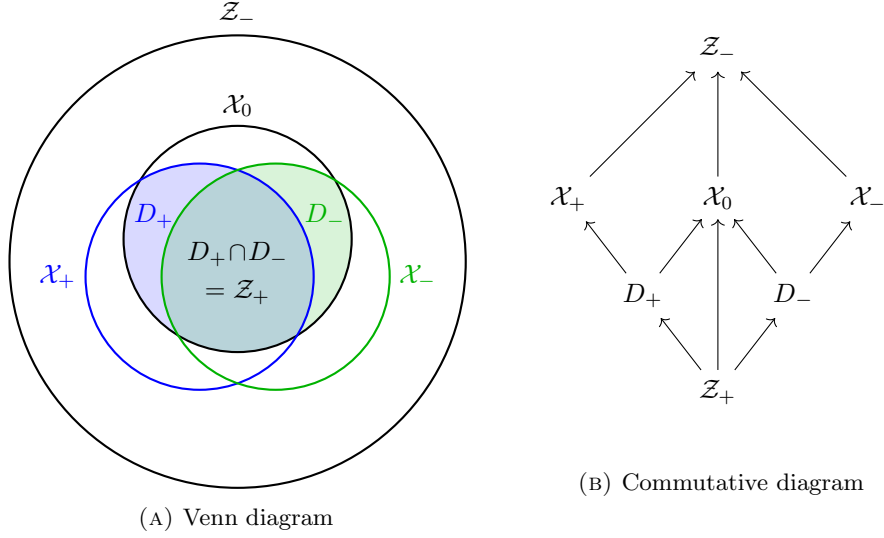
\begin{figure}[ht]
  \centering
  \begin{subfigure}{0.53\textwidth}
    \centering
    \begin{tikzpicture}
      \begin{scope}
        \clip (-0.5,-0.5) circle (1.5);
        \fill[blue,opacity=0.15] (0,0) circle (1.5);
      \end{scope}
      \begin{scope}
        \clip (0.5,-0.5) circle (1.5);
        \fill[black!30!green,opacity=0.15] (0,0) circle (1.5);
      \end{scope}
      \draw[thick](0,0) circle [radius=1.5];
      \draw[thick,blue](-0.5,-0.5) circle [radius=1.5];
      \draw[thick,black!30!green](0.5,-0.5) circle [radius=1.5];
      \node[left,blue] at (-2,-0.5) {$\hs$};
      \node[right,black!30!green] at (2,-0.5) {$\hsdual$};
      \node[above] at (0,1.5) {$\hsmid$};
      \node[blue] at (-1.1,0.3) {$D_{+}$};
      \node[black!30!green] at (1.15,0.3) {$D_{-}$};
      \node at (0,-0.2) {$D_{+}\!\cap\! D_{-}$};
      \node at (0,-0.7) {$= \Zp$};
      \draw[thick](0,-0.3) circle [radius=3];
      \node[above] at (0,2.7) {$\Zm$};
    \end{tikzpicture}
    \caption{Venn diagram}
    \label{fig:embedded-gelfand-triple-venn-diagram}
  \end{subfigure}
  \begin{subfigure}{0.46\textwidth}
    \centering
    \begin{tikzcd}[column sep=0.5em, row sep=2em]
      & & \Zm & & \\
      \\
      \hs\ar{rruu} & & \hsmid\ar{uu} & & \hsdual\ar{lluu} \\
      & D_{+}\ar{ru}\ar{lu} & & D_{-}\ar{lu}\ar{ru} & \\
      & & \Zp\ar{ru}\ar{lu}\ar{uu}& &
    \end{tikzcd}
    \vspace{1.5em}
    \caption{\label{fig:emdedded-gelfand-triple-commutative-diagram}Commutative diagram}
  \end{subfigure}
  \caption{quasi Gelfand triple embedded in $\Zm$}
  \label{fig:embedded-gelfand-triple}
\end{figure}

\begin{lemma}\label{th:Zm-equals-hs-plus-hsdual}
  $\Zm = \hs + \hsdual$ and
  \begin{equation*}
    \norm{h}_{\Zm} = \inf_{\substack{f+g = h \\ f \in \hs, g\in\hsdual}} \sqrt{\norm{f}_{\hs}^{2} + \norm{g}_{\hsdual}^{2}}.
  \end{equation*}
\end{lemma}

\begin{proof}
  Note that $\Zp$ is a Hilbert space with $\scprod{z_{1}}{z_{2}}_{\Zp} = \scprod{z_{1}}{z_{2}}_{\hs} + \scprod{z_{1}}{z_{2}}_{\hsdual}$.
  Hence, there is a duality map $\Phi$ from $\Zm$ to $\Zp$ and we can write
  \begin{equation*}
    \dualprod{h}{z}_{\Zm,\Zp} = \scprod{\Phi h}{z}_{\Zp} = \scprod{\Phi h}{z}_{\hs} + \scprod{\Phi h}{z}_{\hsdual}.
  \end{equation*}
  Furthermore, with the duality map $\Psi$ from $\hsdual$ to $\hs$ we have
  \begin{equation*}
    \dualprod{h}{z}_{\Zm,\Zp} = \dualprod{\Psi^{-1} \Phi h}{z}_{\hsdual,\hs} + \dualprod{\Psi \Phi h}{z}_{\hs,\hsdual}
  \end{equation*}
  and $h = \Psi^{-1}\Phi h + \Psi \Phi h$ in $\Zm$, where $\Psi^{-1}\Phi h \in \hsdual$ and $\Psi\Phi h \in \hs$.

  Let $h \in \Zm$. Then for every $f \in \hs$, $g \in \hsdual$ that satisfy $h = f + g$ in $\Zm$ we have
  \begin{align*}
    \abs{\dualprod{h}{z}_{\Zm,\Zp}}
    &= \abs{\dualprod{f}{z}_{\hs,\hsdual} + \dualprod{g}{z}_{\hsdual,\hs}}
    \leq \abs{\dualprod{f}{z}_{\hs,\hsdual}}
      + \abs{\dualprod{g}{z}_{\hsdual,\hs}}
    \\
    &\leq \norm{f}_{\hs}\norm{z}_{\hsdual} + \norm{g}_{\hsdual} \norm{z}_{\hs}
    \\
    &\leq \sqrt{\norm{f}_{\hs}^{2} + \norm{g}_{\hsdual}^{2}} \sqrt{\norm{z}_{\hsdual}^{2} + \norm{z}_{\hs}^{2}}
    \\
    &= \sqrt{\norm{f}_{\hs}^{2} + \norm{g}_{\hsdual}^{2}} \norm{z}_{\Zp},
  \end{align*}
  which implies $\norm{h}_{\Zm} \leq \inf_{h = f + g} \sqrt{\norm{f}_{\hs}^{2} + \norm{g}_{\hsdual}^{2}}$.
  On the other hand
  \begin{align*}
    \norm{h}_{\Zm}^{2} = \norm{\Phi h}_{\Zp}^{2} = \norm{\Phi h}_{\hs}^{2} + \norm{\Phi h}_{\hsdual}^{2}
    =\norm{\Psi^{-1} \Phi h}_{\hsdual}^{2} + \norm{\Psi \Phi h}_{\hs}^{2}
  \end{align*}
  finishes the proof.
\end{proof}

The next result reinforces \Cref{def:intersections-of-members-of-quasi-Gelfand-triple}.

\begin{proposition}
  The intersection $\hs \cap \hsmid$ in $\Zm$ equals $D_{+}$, i.e., $\ran \psi_{\hs} \cap \ran \psi_{\hsmid} = \ran (\psi_{\hsmid}\circ \iota_{+})$, and the intersection $\hsdual \cap \hsmid$ in $\Zm$ equals $D_{-}$, i.e., $\ran \psi_{\hsdual} \cap \ran \psi_{\hsmid} = \ran (\psi_{\hsmid}\circ \iota_{-})$.
\end{proposition}

\begin{proof}
  Let $h \in \hs\cap\hsmid \subseteq \Zm$, i.e., it exists an $f \in \hs$ and a $k \in \hsmid$ such that
  \begin{equation*}
    \dualprod{h}{z}_{\Zm,\Zp} = \dualprod{f}{\iota_{-}^{-1}z}_{\hs,\hsdual} = \scprod{k}{z}_{\hs}\quad\text{for all}\quad z \in \Zp = D_{+} \cap D_{-}.
  \end{equation*}
  We define $x = \iota_{-}^{-1}z$, which leads to
  \begin{equation*}
    \dualprod{f}{x}_{\hs,\hsdual} = \scprod{k}{\iota_{-}x}_{\hs}\quad\text{for all}\quad x \in \iota_{-}^{-1}(D_{+} \cap D_{-}).
  \end{equation*}
  Since $\iota_{-}^{-1}(D_{+} \cap D_{-})$ is a core of $\iota_{-}$ (\Cref{th:Dplus-Dminus-core-of-iota-adjun-iota}), this equation is also true for all $x \in \dom \iota_{-}$.
  Moreover, this implies $f = \iota_{-}\adjun k$ and $k \in \dom \iota_{-}\adjun = D_{+}$. By $\iota_{-}\adjun = \iota_{+}^{-1}$ we obtain $\iota_{+}f = k \in D_{+}$ and
  \begin{equation*}
    \dualprod{h}{z}_{\Zm,\Zp} = \dualprod{f}{\iota_{-}^{-1} z}_{\hs,\hsdual} = \scprod{k}{z}_{\hsmid} = \scprod{\iota_{+}f}{z}_{\hsmid},
  \end{equation*}
  which gives $h = f = k = \iota_{+} f$ in $\Zm$ and $h \in D_{+}$.

  The same steps can also be done for $\hsdual$.
\end{proof}

\begin{theorem}\label{th:X-plus-cap-X-minus-in-X-0}
  The intersection $\hs\cap \hsdual$ in $\Zm$ is $D_{+}\cap D_{-}(= \Zp)$,
  i.e.,
  \[
    \ran \psi_{\hs} \cap \ran \psi_{\hsdual}
    = \ran (\psi_{\hsmid}\circ \iota_{+}) \cap \ran (\psi_{\hsmid} \circ \iota_{-})
    = \psi_{\hsmid}(\Zp).
  \]
\end{theorem}

This means that area of $\hs \cap \hsdual$ in \Cref{fig:embedded-gelfand-triple} outside of $\hsmid$ is actually empty.

\begin{proof}
  Let $h \in \hs \cap \hsdual \subseteq \Zm$, i.e., it exists an $f \in \hs$ and a $g \in \hsdual$ such that
  \begin{equation*}
    \dualprod{h}{z}_{\Zm,\Zp}
    = \dualprod{f}{\iota_{-}^{-1} z}_{\hs,\hsdual}
    = \dualprod{g}{\iota_{+}^{-1} z}_{\hsdual,\hs}
    \quad\text{for all}\quad
    z \in D_{+}\cap D_{-}.
  \end{equation*}
  We define $x \coloneqq \iota_{+}^{-1}z$, which leads to $z = \iota_{+} x$. Since $z \in \dom \iota_{-}^{-1}$, we have $x\in \dom \iota_{-}^{-1}\iota_{+}$.
  Recall that $\iota_{-}^{-1} = \iota_{+}\adjun$ and $\iota_{+}^{-1} \Zp = \iota_{+}^{-1}(D_{+}\cap D_{-}) = \dom \iota_{+}\adjun \iota_{+}$ 
  (see \Cref{le:iota-plus-adjun=iota-minus-inverse} and \Cref{th:Dplus-Dminus-core-of-iota-adjun-iota}).
  Hence,
  \begin{equation*}
    \dualprod{f}{\iota_{+}\adjun \iota_{+} x}_{\hs,\hsdual}
    = \dualprod{g}{x}_{\hsdual,\hs}
    \quad\text{for all}\quad
    x \in \dom \iota_{+}\adjun \iota_{+},
  \end{equation*}
  which implies $(\iota_{+}\adjun \iota_{+})\adjun f = g$ and $f \in \dom (\iota_{+}\adjun\iota_{+})\adjun$. By \Cref{th:TadjunT-selfadjoint-for-dual-pair} $(\iota_{+}\adjun \iota_{+})\adjun = \iota_{+}\adjun \iota_{+}$ and therefore $f \in \dom \iota_{+}\adjun \iota_{+}$ and in particular, $\iota_{+} f \in \iota_{+}(\dom \iota_{+}\adjun \iota_{+}) = D_{+}\cap D_{-}$.
  Note that again by $\iota_{-}^{-1} = \iota_{+}\adjun$ we have $\iota_{-}^{-1}\iota_{+} f = g$.
  Thus, $g \in \dom \iota_{-}$ and $\iota_{+} f = \iota_{-} g$. This gives
  \begin{equation*}
    \dualprod{h}{z}_{\Zm,\Zp} = \scprod{\iota_{+}f}{z}_{\hsmid} = \scprod{\iota_{-} g}{z}_{\hsmid}.
  \end{equation*}
  Therefore, $h = f = g = \iota_{+} f = \iota_{-} g$ in $\Zm$.
\end{proof}

\begin{corollary}\label{th:norm-in-Zm-equals-inf}
  For $f \in \hs$ and $g \in \hsdual$ we have
  \begin{align*}
    \norm{f + g}_{\Zm} = \inf_{z \in \Zp} \sqrt{\norm{f + z}_{\hs}^{2} + \norm{g - z}_{\hsdual}^{2}}.
  \end{align*}
\end{corollary}

\begin{proof}
  By \Cref{th:Zm-equals-hs-plus-hsdual} we have
  \begin{align*}
    \norm{f+g}_{\Zm} = \inf_{\substack{\tilde{f}+\tilde{g}=f+g \\ \tilde{f} \in \hs, \tilde{g} \in \hsdual}} \sqrt{\norm{\tilde{f}}_{\hs}^{2} + \norm{\tilde{g}}_{\hsdual}^{2}}
  \end{align*}
  Note that $f + g = \tilde{f} + \tilde{g}$ implies
  \begin{align*}
    z \coloneqq \underbrace{f - \tilde{f}}_{\in\mathrlap{\hs}} = - \underbrace{(g - \tilde{g})}_{\in\mathrlap{\hsdual}} \in \hs\cap\hsdual.
  \end{align*}
  We can write $\tilde{f} = f - z$ and $\tilde{g} = f + z$ and by \Cref{th:X-plus-cap-X-minus-in-X-0} we have $z \in \Zp$. Consequently,
  \begin{equation*}
    \norm{f + g} = \inf_{z \in \Zp} \sqrt{\norm{f - z}_{\hs}^{2} + \norm{g + z}_{\hsdual}^{2}}
    = \inf_{z \in \Zp} \sqrt{\norm{f + z}_{\hs}^{2} + \norm{g - z}_{\hsdual}^{2}}.
    \ifSn{\tag*{\qedhere}}{\qedhere}
  \end{equation*}
\end{proof}

The space $\Zm$ is the smallest space where we can embed the quasi Gelfand triple structure preservingly.
The following theorem makes this statement precise.

\begin{theorem}\label{th:Zm-is-an-initial-object}
  Let $\mathcal{H}$ be a Hausdorff topological vector space such that we can structure preservingly embed the quasi Gelfand triple $(\hs,\hsmid,\hsdual)$ into $\mathcal{H}$ and let $\phi_{\hs}$, $\phi_{\hsmid}$ and $\phi_{\hsdual}$ denote the embeddings.
  Then also $\Zm$ can be continuously embedded into $\mathcal{H}$ by a mapping $\phi_{\Zm}$, such that
  \begin{equation*}
    \phi_{\Zm}\circ \psi_{\hs} = \phi_{\hs},
    \quad
    \phi_{\Zm}\circ \psi_{\hsmid} = \phi_{\hsmid}
    \quad\text{and}\quad
    \phi_{\Zm}\circ \psi_{\hsdual} = \phi_{\hsdual},
  \end{equation*}
  i.e., the following diagram commutes.
  
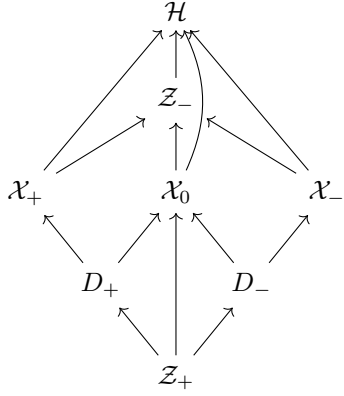
\begin{equation*}
    \begin{tikzcd}[column sep=0.5em]
      & & \mathcal{H} & & \\
      & & \Zm \ar{u} & & \\
      \hs\ar{rru}\ar{rruu} & & \hsmid\ar{u}\ar[bend right=25]{uu} & & \hsdual\ar{llu}\ar{lluu} \\
      & D_{+}\ar{ru}\ar{lu} & & D_{-}\ar{lu}\ar{ru} & \\
      & & \Zp\ar{ru}\ar{lu}\ar{uu}& &
    \end{tikzcd}
  \end{equation*}
\end{theorem}

\begin{proof}
  \newcommand{\haths}{\hat{\mathcal{X}}_{+}}
  \newcommand{\hathsmid}{\hat{\mathcal{X}}_{0}}
  \newcommand{\hathsdual}{\hat{\mathcal{X}}_{-}}
  \newcommand{\hatZp}{\hat{\mathcal{Z}}_{+}}
  \newcommand{\hatZm}{\hat{\mathcal{Z}}_{-}}
  Recall that we can assume that $\hs,\hsmid,\hsdual \subseteq \Zm$ and $\psi_{\hs}f = f$, $\psi_{\hsmid} h = h$, and $\psi_{\hsdual} g = g$, by simply replacing the quasi Gelfand triple $(\hs,\hsmid,\hsdual)$ by $(\psi_{\hs}(\hs),\psi_{\hsmid}(\hsmid), \psi_{\hsdual}(\hsdual))$, see \Cref{fig:embedded-gelfand-triple}.

  For convenience we define $\haths = \phi_{\hs}(\hs)$, $\hathsmid = \phi_{\hsmid}(\hsmid)$ and $\hathsdual = \phi_{\hsdual}(\hsdual)$ with $\norm{f}_{\haths} = \norm{\phi_{\hs}^{-1} f}_{\hs}$, $\norm{h}_{\hathsmid} = \norm{\phi_{\hsmid}^{-1} h}_{\hsmid}$ and $\norm{g}_{\hathsdual} = \norm{\phi_{\hsdual}^{-1} g}_{\hsdual}$.

  We will show as a first step that we can endow $\haths + \hathsdual$ in $\mathcal{H}$ with $\norm{h}_{\haths + \hathsdual} = \inf_{f+g=h} \sqrt{\norm{f}_{\haths}^{2} + \norm{g}_{\hathsdual}^{2}}$ such that the corresponding topology of $\norm{\cdot}_{\haths + \hathsdual}$ is finer than the topology $\mathcal{T}_{\mathcal{H}}$ of $\mathcal{H}$ (i.e., whenever $(h_{n})_{n\in\N}$ converges w.r.t.\ $\norm{\cdot}_{\haths + \hathsdual}$, it also converges w.r.t.\ $\mathcal{T}_{\mathcal{H}}$).
  Note that we can alternatively write the norm as
  \begin{align*}
    \norm{f + g}_{\haths + \hathsdual}
    &= \inf \dset*{\sqrt{\norm{\tilde{f}}_{\haths}^{2} + \norm{\tilde{g}}_{\hathsdual}^{2}}}{\tilde{f} + \tilde{g} = f + g} \\
    &= \inf \dset*{\sqrt{\norm{f + z}_{\haths}^{2} + \norm{g - z}_{\hathsdual}^{2}}}{z \in \hs \cap \hsdual}.
  \end{align*}
  Moreover, the mapping
  \begin{align*}
    \Lambda\colon \mapping{\hs \times \hsdual}{\mathcal{H}}{
    \begin{bsmallmatrix}
      f \\ g
    \end{bsmallmatrix}
    }{\phi_{\hs}f + \phi_{\hsdual}g,}
  \end{align*}
  is continuous as composition of the continuous embeddings into $\mathcal{H}$ and the continuous addition in $\mathcal{H}$. Hence, $\ker \Lambda$ is closed in $\hs\times\hsdual$ and the quotient space $\quspace{\hs\times\hsdual}{\ker \Lambda}$ is a Hilbert space and is isometrically isomorphic to $\haths + \hathsdual$ with $\norm{\cdot}_{\haths + \hathsdual}$.
  The quotient mapping $\quspace{\Lambda}{\ker \Lambda}\colon \quspace{\hs \times \hsdual}{\ker \Lambda} \to \mathcal{H}$ is injective and continuous, which implies that topology of $\norm{\cdot}_{\haths+\hathsdual}$ is finer than the trace topology of $\mathcal{T}_{\mathcal{H}}$ on $\haths + \hathsdual$.

  We can regard $\hatZp \coloneqq \phi_{\hsmid} (\Zp) \subseteq \hathsmid \subseteq \mathcal{H}$ and endow this space with
  \begin{align*}
    \norm{z}_{\hatZp}
    \coloneqq \sqrt{\norm{z}_{\haths}^{2} + \norm{z}_{\hathsdual}^{2}} = \norm{\phi_{\hsmid}^{-1} z}_{\Zp}
    \quad\text{for}\quad z \in \hatZp.
  \end{align*}
  Furthermore, we can define a new norm on $\hsmid$ by
  \(
  \norm{h}_{\hatZm} \coloneqq \sup_{z \in \hatZp\setminus\sset{0}}\frac{\abs{\scprod{h}{z}}_{\hathsmid}}{\norm{z}_{\hatZp}}
  \).
  Note that every $h \in \hsmid$ can be written as $h = f + g$, where $f \in D_{+}$ and $g \in D_{-}$, see \Cref{co:Dplus-plus-Dminus-pivot-space}.
  Hence, also every $h \in \hathsmid$ can be written as $h = f + g$, where $f \in \phi_{\hsmid}(D_{+}) = \phi_{\hs}(D_{+}) \subseteq \haths \cap \hathsmid$ and $g \in \phi_{\hsmid}(D_{-}) = \phi_{\hsdual}(D_{-}) \subseteq \hathsmid \cap \hathsdual$.
  We know by \Cref{th:norm-in-Zm-equals-inf} for every $f + g \in \hathsmid$ ($f\in \phi_{\hsmid}(D_{+})$, $g\in \phi_{\hsmid}(D_{-})$) that
  \begin{align*}
    \norm{f + g}_{\hatZm}
    &= \norm{\phi_{\hsmid}^{-1}(f + g)}_{\Zm}
      = \inf_{z \in \Zp} \sqrt{\norm{\phi_{\hsmid}^{-1}(f) + z}_{\hs}^{2} + \norm{\phi_{\hsmid}^{-1}(g) - z}_{\hsdual}^{2}}
    \\
    &= \inf_{z \in \hatZp} \sqrt{\vphantom{\norm{\cdot}_{\hs}}
      \smash[b]{
      \norm{\underbrace{\phi_{\hsmid}^{-1}(f) + \phi_{\hsmid}^{-1}(z)}_{=\mathrlap{\phi_{\hs}^{-1}(f + z)}}}_{\hs}^{2}
      + \norm{\underbrace{\phi_{\hsmid}^{-1}(g) - \phi_{\hsmid}^{-1}(z)}_{=\mathrlap{\phi_{\hsdual}^{-1}(g - z)}}}_{\hsdual}^{2}}}
      \vphantom{\norm{\underbrace{\phi_{\hsmid}^{-1}(f) + \phi_{\hsmid}^{-1}(z)}_{=\mathrlap{\phi_{\hs}^{-1}(f + z)}}}_{\hs}^{2}}
    \\
    &= \inf_{z \in \hatZp} \sqrt{\norm{f + z}_{\haths}^{2} + \norm{g - z}_{\hathsdual}^{2}}
    \\
    &\geq \inf_{z \in \hs\cap\hsdual} \sqrt{\norm{f + z}_{\haths}^{2} + \norm{g - z}_{\hathsdual}^{2}}
      = \norm{f + g}_{\haths + \hathsdual},
  \end{align*}
  because $\hatZp \subseteq \haths \cap \hathsdual$. Hence, the completion of $\hathsmid$ w.r.t.\ $\norm{\cdot}_{\hatZm}$ can also be continuously embedded into $\haths + \hathsdual$, because $\haths + \hathsdual$ is complete, and therefore also into $\mathcal{H}$. In particular the mapping ($\psi_{\hsmid}$ does not do anything by assumption)
  \begin{equation*}
    \phi_{\hsmid} \circ \psi_{\hsmid}^{-1} \colon \psi_{\hsmid}(\hsmid) \subseteq \Zm \to \mathcal{H}
  \end{equation*}
  is continuous w.r.t.\ the $\norm{\cdot}_{\Zm}$ topology on $\psi_{\hsmid}(\hsmid)$ and $\mathcal{T}_{\mathcal{H}}$ on $\mathcal{H}$ and injective. By the density of $\hsmid$ in $\Zm$ we can continuously extend this mapping, denoted by
  \begin{equation*}
    \phi_{\Zm} \coloneqq \cl{\phi_{\hsmid} \circ \psi_{\hsmid}^{-1}} \colon \Zm \to \mathcal{H}.
  \end{equation*}
  By construction we already have $\phi_{\Zm} \circ \psi_{\hsmid} = \phi_{\hsmid}$. Note that for $z \in \Zp$ we have
  \begin{align*}
    z = \psi_{\hs} z = \psi_{\hsmid} z = \psi_{\hsdual} z
    \quad\text{and}\quad
    \phi_{\hs} z = \phi_{\hsmid} z = \phi_{\hsdual} z.
  \end{align*}

  Now for $f \in \hs$ there exists a sequence $(z_{n})_{n\in\N}$ in $\Zp$ that converges to $f$ w.r.t.\ $\norm{\cdot}_{\hs}$. Hence, the continuity of $\phi_{\Zm}$, $\psi_{\hs}$ and $\phi_{\hs}$ gives
  \begin{multline*}
    \phi_{\Zm} \circ \psi_{\hs}f = \lim_{n\to\infty} \phi_{\Zm} \circ \psi_{\hs} z_{n}
    = \lim_{n\to\infty} \phi_{\Zm} \circ \psi_{\hsmid} z_{n} \\
    = \lim_{n\to\infty} \phi_{\hsmid}  z_{n}
    = \lim_{n\to\infty} \phi_{\hs}  z_{n}
    =  \phi_{\hs} f.
  \end{multline*}
  Analogously, we can show $\phi_{\Zm} \circ \psi_{\hsdual} = \phi_{\hsdual}$.
\end{proof}

\begin{corollary}
  Let $\mathcal{H}$ be a topological Hausdorff vector space such that we can structure preservingly embed the quasi Gelfand triple $(\hs,\hsmid,\hsdual)$ into $\mathcal{H}$. Then $\hs \cap \hsdual$ in $\mathcal{H}$ equals $D_{+} \cap D_{-}$, i.e., $\phi_{\hs}(\hs) \cap \phi_{\hsdual}(\hsdual) = \phi_{\hsmid}(D_{+} \cap D_{-})$.
\end{corollary}

\begin{proof}
  By \Cref{th:Zm-is-an-initial-object} we can also embed $\Zm$ into $\mathcal{H}$ such that
  \begin{align*}
    \begin{array}{l}
      \hs \\ \hsdual
    \end{array}
    \subseteq \Zm \subseteq \mathcal{H}.
  \end{align*}
  Hence, $\hs \cap \hsdual$ in $\mathcal{H}$ is the same as $\hs \cap \hsdual$ in $\Zm$, which equals, by \Cref{th:X-plus-cap-X-minus-in-X-0}, $D_{+} \cap D_{-} = \Zp$.
\end{proof}

\section{Gram operators}%
\label{sec:gram-operator}

Every quasi Gelfand triple $(\hs,\hsmid,\hsdual)$ is fully determined (up to isomorphic identifications) by $\hsmid$, $\ran \iota_{+}$ and $\norm{\cdot}_{\hs}$ on $\ran \iota_{+}$ (or $\ran \iota_{-}$ with $\norm{\cdot}_{\hsdual}$).
However, in the Hilbert space case ($\hs$ is a Hilbert space) we can even encode the entire information of a quasi Gelfand triple in a single (so called \emph{Gram}) operator $G$ on $\hsmid$, that is self-adjoint, positive and injective. This means that $\scprod{Gf}{g}_{\hsmid}$ defines a new inner product on $\hsmid$, which gives rise to $\scprod{f}{g}_{\hs}$. In particular, we will see that $D_{+} = \dom G^{\nicefrac{1}{2}}$ and $\scprod{G^{\nicefrac{1}{2}}f}{G^{\nicefrac{1}{2}}g} = \scprod{f}{g}_{\hs}$.

\begin{definition}
  Let $(\hs,\hsmid,\hsdual)$ be a quasi Gelfand triple of Hilbert spaces.
  Then we define the \emph{Gram operator} $G_{+} \colon \dom G_{+} \subseteq \hsmid \to \hsmid$ of the quasi Gelfand triple by
  \begin{equation*}
    G_{+} \coloneqq (\iota_{+}^{-1})\hadjun \iota_{+}^{-1} = (\iota_{+}\iota_{+}\hadjun)^{-1},
  \end{equation*}
  where here the adjoint is taken w.r.t.\ the dual pairs $(\hsmid,\hsmid)$ and $(\hs,\hs)$, i.e., $(\iota_{+}^{-1})\hadjun = (\iota_{+}^{-1})\adjunX{\hs\times\hsmid}$ and $\iota_{+}\hadjun = \iota_{+}\adjunX{\hsmid\times\hs}$.
\end{definition}

By \Cref{th:TTadjun-self-adjoint} $G_{+}$ is self-adjoint and positive (not necessarily strictly positive (coercive)). Moreover, by the functional calculus for unbounded self-adjoint operators on Hilbert spaces there exists a root $G_{+}^{\nicefrac{1}{2}}$ of $G_{+}$, which is also self-adjoint and positive.

Clearly, we can do the same for $\iota_{-}$ and define $G_{-} \coloneqq (\iota_{-}^{-1})\hadjun \iota_{-}^{-1}$, where again here the adjoint is taken w.r.t.\ the dual pairs $(\hsmid,\hsmid)$ and $(\hsdual,\hsdual)$, i.e., $(\iota_{-}^{-1})\hadjun = (\iota_{-}^{-1})\adjunX{\hsdual\times\hsmid}$.
In fact we will see that $G_{-} = G_{+}^{-1}$.

\begin{theorem}\label{th:Gplus-defines-hs-inner-product}
  Let $(\hs,\hsmid,\hsdual)$ be a quasi Gelfand triple of Hilbert spaces and $G_{+}$ its Gram operator. Then $\ran \iota_{+} = \dom G_{+}^{\nicefrac{1}{2}}$ and
  \begin{equation*}
    \scprod{f}{g}_{\hs} = \scprod{G_{+}^{\nicefrac{1}{2}} f}{G_{+}^{\nicefrac{1}{2}}g}_{\hsmid}
    \quad\text{for all}\quad f,g \in \dom G_{+}^{\nicefrac{1}{2}}.
  \end{equation*}
  In particular, $\norm{f}_{\hs} = \norm{G_{+}^{\nicefrac{1}{2}}f}_{\hsmid}$.
\end{theorem}

\begin{proof}
  Note that $\dom G_{+} = \dom (\iota_{+}^{-1})\hadjun \iota_{+}^{-1}$ is a core of $\iota_{+}^{-1}$. This implies that for every $f \in \ran \iota_{+}$ there exists a sequence $(f_{n})_{n\in\N}$ in $\dom G_{+}$ such that $f_{n} \to f$ w.r.t.\ $\norm{\cdot}_{\hsmid}$ and $\iota_{+}^{-1} f_{n} \to \iota_{+}^{-1} f$ w.r.t.\ $\norm{\cdot}_{\hs}$. In order words $\dom G_{+}$ is dense in $D_{+}$ w.r.t.\ $\norm{\cdot}_{\hs\cap\hsmid}$.
  For $f,g \in \dom G_{+} \subseteq \dom G_{+}^{\nicefrac{1}{2}}$ we have
  \begin{align}\label{eq:inner-product-equalities-for-dom-G}
    \scprod{\iota_{+}^{-1} f}{\iota_{+}^{-1}g}_{\hs} = \scprod*{(\iota_{+}^{-1})\hadjun \iota_{+}^{-1} f}{g}_{\hs} =
    \scprod{G_{+} f}{g}_{\hsmid} = \scprod{G_{+}^{\nicefrac{1}{2}}f}{G_{+}^{\nicefrac{1}{2}}g}_{\hsmid}
  \end{align}
  and in particular we have $\norm{\iota_{+}^{-1}f}_{\hs} = \norm{G_{+}^{\nicefrac{1}{2}}f}_{\hsmid}$ for all $f \in \dom G_{+}$.

  For every $f \in \ran \iota_{+}$  there exists a sequence $(f_{n})_{n\in\N}$ in $\dom G_{+}$ that converges to $f$ w.r.t.\ $\norm{\cdot}_{\hs\cap\hsmid}$.
  Hence, we have
  \begin{equation*}
    \norm{G_{+}^{\nicefrac{1}{2}} f_{n}}_{\hsmid} = \norm{\iota_{+}^{-1} f_{n}}_{\hs} \to \norm{\iota_{+}^{-1} f}_{\hs}
  \end{equation*}
  and in particular $(G_{+}^{\nicefrac{1}{2}}f_{n})_{n\in\N}$ is a bounded sequence in $\hsmid$. Therefore, there exists a weakly convergent subsequence and by taking a convex combination \Cref{th:weak-to-strong-convergent} we end up with a sequence $(\tilde{f}_{n})_{n\in\N}$ that still converges to $f$ w.r.t.\ $\norm{\cdot}_{\hs\cap\hsmid}$ and additionally $(G_{+}^{\nicefrac{1}{2}}\tilde{f}_{n})_{n\in\N}$ converges to some $\tilde{f} \in \hsmid$. By the closedness of $G_{+}^{\nicefrac{1}{2}}$ the limit $\tilde{f}$ has to coincide with $f$. This implies $\ran \iota_{+} \subseteq \dom G_{+}$ and we can extend~\eqref{eq:inner-product-equalities-for-dom-G} by continuity to
  \begin{equation*}
    \scprod{\iota_{+}^{-1} f}{\iota_{+}^{-1}g}_{\hs} = \scprod[\big]{G_{+}^{\nicefrac{1}{2}}f}{G_{+}^{\nicefrac{1}{2}}g}_{\hsmid}
    \quad\text{for all}\quad f,g \in \ran \iota_{+}.
  \end{equation*}

  Note that $\dom G_{+} = \dom (G_{+}^{\nicefrac{1}{2}})\adjun G_{+}^{\nicefrac{1}{2}}$ is a core of $G_{+}^{\nicefrac{1}{2}}$. Now we will the repeat the previous step with switched roles of $G_{+}^{\nicefrac{1}{2}}$ and $\iota_{+}^{-1}$: For every $f \in \dom G_{+}^{\nicefrac{1}{2}}$ there exists a sequence $(f_{n})_{n\in\N}$ in $\dom G_{+}$ such that $f_{n} \to f$ and $G_{+}^{\nicefrac{1}{2}}f_{n} \to G_{+}^{\nicefrac{1}{2}}f$ both w.r.t.\ $\norm{\cdot}_{\hs}$. This gives
  \begin{align*}
    \norm{\iota_{+}^{-1} f_{n}}_{\hs} = \norm{G_{+}^{\nicefrac{1}{2}}f_{n}}_{\hsmid} \to \norm{G_{+}^{\nicefrac{1}{2}}f}_{\hsmid}.
  \end{align*}
  Now $(\iota_{+}^{-1}f_{n})_{n\in\N}$ is a bounded sequence in $\hs$. Therefore there exists a weakly convergent subsequence. Moreover a convex combination of this subsequence converges even w.r.t\ $\norm{\cdot}_{\hs}$.
  In total we have a sequence $(\tilde{f}_{n})_{n\in\N}$ such that $\tilde{f}_{n} \to f$, $G_{+}^{\nicefrac{1}{2}}\tilde{f}_{n} \to G_{+}^{\nicefrac{1}{2}}f$ w.r.t.\ $\norm{\cdot}_{\hsmid}$ and $\iota_{+}^{-1}\tilde{f}_{n} \to \tilde{f}$ w.r.t.\ $\norm{\cdot}_{\hs}$ for an $\tilde{f} \in \hs$.
  By the closedness of $\iota_{+}$ we conclude $\tilde{f} = \iota_{+}^{-1}f$ and in turn $\dom G_{+}^{\nicefrac{1}{2}} \subseteq \ran \iota_{+}$, which completes the proof.
\end{proof}

\begin{proposition}\label{th:Gminus-equals-Gplus-inverse}
  Let $(\hs,\hsmid,\hsdual)$ be a quasi Gelfand triple of Hilbert spaces. Then
  \begin{equation*}
    G_{-} = G_{+}^{-1}.
  \end{equation*}
\end{proposition}

\begin{proof}
  Let $\Psi\colon \hsdual \to \hs$ denote the duality mapping between $\hsdual$ and $\hs$.
  Recall $G_{-} = (\iota_{-}^{-1})\hadjun \iota_{-}^{-1} = (\iota_{-} \iota_{-}\hadjun)^{-1}$,
  \begin{equation*}
    \iota_{+}\hadjun = \Psi\iota_{+}\adjun = \Psi \iota_{-}^{-1}
    \quad\text{and}\quad
    \iota_{-}\hadjun = \Psi^{-1} \iota_{-}\adjun = \Psi^{-1} \iota_{+}^{-1}.
  \end{equation*}
  Hence, we have
  \begin{equation*}
    G_{-}^{-1} = \iota_{-} \iota_{-}\hadjun = \iota_{-} \Psi^{-1} \iota_{+}^{-1} = (\Psi \iota_{-}^{-1})^{-1}\iota_{+}^{-1} = (\iota_{+}\hadjun)^{-1} \iota_{+}^{-1} = (\iota_{+}^{-1})\hadjun \iota_{+}^{-1} = G_{+}.
    \ifSn{\tag*{\qedhere}}{\qedhere}
  \end{equation*}
\end{proof}

\begin{corollary}
  Let $(\hs,\hsmid,\hsdual)$ be a quasi Gelfand triple of Hilbert spaces. Then
  \begin{equation*}
    \ran \iota_{-}  = \dom G_{-}^{\nicefrac{1}{2}} = \dom G_{+}^{-\nicefrac{1}{2}} = \ran G_{+}^{\nicefrac{1}{2}}.
  \end{equation*}
\end{corollary}

So far we have shown that there is a self-adjoint positive and injective operator with dense range for every quasi Gelfand triple.
Now the next theorem will show that also the reverse is true. That is, every self-adjoint positive and injective operator $G$ with dense range establishes a quasi Gelfand triple whose Gram operator is $G$.

\begin{theorem}\label{th:Gram-operator-gives-quasi-Gelfand-triple}
  Let $\hsmid$ be a Hilbert space and $G$ a self-adjoint positive and injective operator on $\hsmid$ with dense range.
  Then there exists a quasi Gelfand triple whose Gram operator is $G$. In particular, if we denote the corresponding quasi Gelfand triple by $(\hs,\hsmid,\hsdual)$ we have
  \begin{align*}
    \ran \iota_{+} = \dom G^{\nicefrac{1}{2}} \quad\text{and}\quad \ran \iota_{-} = \ran G^{\nicefrac{1}{2}}.
  \end{align*}
  Moreover, $G$ coincides with the Gram operator $G_{+}$ of $(\hs,\hsmid\,\hsdual)$, i.e., $G = G_{+}$.
\end{theorem}

Note that dense range and injectivity are for a self-adjoint operator equivalent. Moreover, the density of the range (or the injectivity of the operator) is not really a necessity as we can always split
\begin{align*}
  \hsmid = \ker G \oplus \cl{\ran G}.
\end{align*}
Hence, we just replace $\hsmid$ with $\cl{\ran G}$ and $G$ with $G\big\vert_{\cl{\ran G}}$.

\begin{proof}
  We define $\scprod{f}{g}_{\hs} \coloneqq \scprod{G^{\nicefrac{1}{2}} f}{G^{\nicefrac{1}{2}} g}_{\hsmid}$ and the corresponding norm $\norm{f}_{\hs} = \norm{G^{\nicefrac{1}{2}}f}_{\hsmid}$ for $f,g \in \dom G^{\nicefrac{1}{2}}$. Since $G^{\nicefrac{1}{2}}$ is positive $\scprod{\cdot}{\cdot}_{\hs}$ is really an inner product and $\norm{\cdot}_{\hs}$ a norm. Hence, $\dom G^{\nicefrac{1}{2}}$ with $\scprod{\cdot}{\cdot}_{\hs}$ is a pre-Hilbert space and its completion $\hs$ is a Hilbert space. We define
  \begin{equation*}
    \iota_{+}\colon
    \mapping{\dom G^{\nicefrac{1}{2}}\subseteq \hs}{\hsmid}{f}{f.}
  \end{equation*}
  Let
  \(\left(
  \begin{bsmallmatrix}
    f_{n} \\ f_{n}
  \end{bsmallmatrix}
  \right)_{n\in\N}
  \)
  be a sequence in $\iota_{+}$ that converges to
  \(
  \begin{bsmallmatrix}
    g \\ f
  \end{bsmallmatrix}
  \in \hs \times \hsmid
  \).
  Then
  \(
  \left(
  \begin{bsmallmatrix}
    f_{n} \\ G^{\nicefrac{1}{2}} f_{n}
  \end{bsmallmatrix}
  \right)_{n\in\N}
  \)
  is a Cauchy sequence in $\hsmid \times \hsmid$, and therefore convergent. The closedness of $G^{\nicefrac{1}{2}}$ implies $f \in \dom G^{\nicefrac{1}{2}} = D_{+}$ and
  \(
  \begin{bsmallmatrix}
    f_{n} \\ G^{\nicefrac{1}{2}} f_{n}
  \end{bsmallmatrix}
  \to
  \begin{bsmallmatrix}
    f \\ G^{\nicefrac{1}{2}} f
  \end{bsmallmatrix}
  \).
  This leads to $\norm{f_{n} - f}_{\hs} = \norm{G^{\nicefrac{1}{2}}(f_{n}- f)}_{\hsmid} \to 0$ and consequently $f = g$.
  Now we can apply \Cref{th:quasi-gelfand-triple-charaterization} and see that there is a space $\hsdual$ such that $(\hs,\hsmid,\hsdual)$ forms a quasi Gelfand triple.


  Now we have for $f,g \in \dom G^{\nicefrac{1}{2}} = \ran \iota_{+} = \dom G_{+}^{\nicefrac{1}{2}}$
  \begin{align*}
    \scprod{G^{\nicefrac{1}{2}} f}{G^{\nicefrac{1}{2}} g}_{\hsmid}
    = \scprod{f}{g}_{\hs}
    = \scprod{G_{+}^{\nicefrac{1}{2}} f}{G_{+}^{\nicefrac{1}{2}} g}_{\hsmid}.
  \end{align*}
  Note that $\dom G \subseteq \dom G^{\nicefrac{1}{2}}$ and therefore for $f \in \dom G$ we have
  \begin{align*}
    \scprod{Gf}{g}_{\hsmid} = \scprod{G_{+}^{\nicefrac{1}{2}} f}{G_{+}^{\nicefrac{1}{2}}g}_{\hsmid},
  \end{align*}
  which implies $G_{+}^{\nicefrac{1}{2}} f \in \dom G_{+}^{\nicefrac{1}{2}}$ and $G_{+}^{\nicefrac{1}{2}} G_{+}^{\nicefrac{1}{2}} f = G f$. Hence $G \subseteq G_{+}$. The same argument with $G$ and $G_{+}$ switched gives $G_{+} \subseteq G$ and thus $G = G_{+}$.

  By \Cref{th:Gminus-equals-Gplus-inverse} we have $G_{-} = G_{+}^{-1} = G^{-1}$ and therefore, by \Cref{th:Gplus-defines-hs-inner-product} for $G_{-}$,
  \begin{equation*}
    \ran \iota_{-} = \dom G_{-}^{\nicefrac{1}{2}} = \ran G_{-}^{-\nicefrac{1}{2}} = \ran G^{\nicefrac{1}{2}}.
    \ifSn{\tag*{\qedhere}}{\qedhere}
  \end{equation*}
\end{proof}

There is a bijection between the set of quasi Gelfand triples with pivot space $\hsmid$ and all self-adjoint positive and injective operators with dense range on $\hsmid$, see \Cref{fig:bijection-quasi-Gelfand-triples-Gram-operators}.

\begin{figure}[ht]
  \centering
  \begin{tikzcd}[column sep=2em,row sep=2em]
    (\hs,\hsmid,\hsdual) \ar[out=80,in=100]{rrr}{(\iota_{+}\iota_{+}\adjun)^{-1}}& & & G\ar[outer sep=-3pt,end anchor=east,out=-100,in=0]{ld}{\begin{array}{ll}\dom G^{\nicefrac{1}{2}}, \\ \scprod{G^{\nicefrac{1}{2}}\cdot}{G^{\nicefrac{1}{2}}\cdot}_{\hsmid}\end{array}} \\
    & \hs,\iota_{+}\ar[start anchor=west,out=180, in=-80]{lu}{\text{\Cref{th:quasi-gelfand-triple-charaterization}}} & D_{+},\scprod{\cdot}{\cdot}_{\hs}\ar[outer sep=3pt]{l}{\text{completion}} &
  \end{tikzcd}
  \caption{\label{fig:bijection-quasi-Gelfand-triples-Gram-operators}Illustration of \Cref{th:Gram-operator-gives-quasi-Gelfand-triple}}
\end{figure}

Since all infinite dimensional separable Hilbert spaces are isomorphic, it is clear that there exists a dual pairing $\dualprod{\cdot}{\cdot}_{\hs,\hsmid}$ such that also $(\hs,\hsmid)$ is a complete dual pair. However, we can explicitly write this mapping by
\begin{equation*}
  \dualprod{f}{g}_{\hs,\hsmid} = \scprod[\Big]{\cl{G_{+}^{\nicefrac{1}{2}}\iota_{+}}f}{g}_{\hsmid} = \scprod[\Big]{f}{\cl{G_{+}^{\nicefrac{1}{2}}\iota_{+}}^{-1} g}_{\hs},
\end{equation*}
where $\cl{G_{+}^{\nicefrac{1}{2}}\iota_{+}}$ is the continuous extension of the isometric mapping $G_{+}^{\nicefrac{1}{2}}\iota_{+}\colon \dom \iota_{+} \subseteq \hs \to \hsmid$.

\subsection{Decomposition into two ``ordinary'' Gelfand triples}

In this section we will see that every quasi Gelfand triple of Hilbert spaces can be decomposed into two ``ordinary'' Gelfand triple. This means for a quasi Gelfand triple $(\hs,\hsmid,\hsdual)$ there exist ``ordinary'' Gelfand triples $\hs^{1} \subseteq \hsmid^{1} \subseteq \hsdual^{1}$ and $\hs^{2} \subseteq \hsmid^{2} \subseteq \hsdual^{2}$ such that
\begin{equation*}
  \hs = \hs^{1} \oplus \hsdual^{2},\quad
  \hsmid = \hsmid^{1} \oplus \hsmid^{2}
  \quad\text{and}\quad
  \hsdual = \hsdual^{1} \oplus \hs^{2}.
\end{equation*}

\begin{theorem}\label{th:decomposition-into-ordinary-Gelfand-triple}
  Let $(\hs,\hsmid,\hsdual)$ be a quasi Gelfand triple of Hilbert spaces. Then there exist two ``ordinary'' Gelfand triple $\hs^{1} \subseteq \hsmid^{1} \subseteq \hsdual^{1}$ and $\hs^{2} \subseteq \hsmid^{2} \subseteq \hsdual^{2}$ such that
  \begin{equation*}
    \hs = \hs^{1} \oplus \hsdual^{2},\quad
    \hsmid = \hsmid^{1} \oplus \hsmid^{2}
    \quad\text{and}\quad
    \hsdual = \hsdual^{1} \oplus \hs^{2}.
  \end{equation*}
\end{theorem}

This means that every quasi Gelfand triple (of Hilbert spaces) is the result of two ``ordinary'' Gelfand triple that are cross-wise composed.

\begin{proof}
  We will step the proof in several steps:

  \begin{enumerate}[label=\textup{\textbf{\arabic{*}.~Step:}},leftmargin=0pt, align=left, itemindent=*, labelsep=3pt, itemsep=3pt, parsep=2pt]
    \item \emph{Decomposition of $\hsmid$.}
          Let $G_{+}$ be the Gram operator of the quasi Gelfand triple and $G_{+}^{\nicefrac{1}{2}}$ its positive square root. Then there exists a spectral measure $E$ for $G_{+}^{\nicefrac{1}{2}}$ such that $G_{+}^{\nicefrac{1}{2}} = \int_{\R_{+}} \lambda \dx[E(\lambda)]$. We can decompose $\hsmid$ into
          \begin{equation*}
            \hsmid =
            \underbrace{\ran E((1,\infty))}_{\eqqcolon\mathrlap{\hsmid^{1}}}
            \mathclose{} \oplus \mathopen{}
            \underbrace{\ran E((0,1])}_{\eqqcolon\mathrlap{\hsmid^{2}}}.
          \end{equation*}
          By spectral theory $\hsmid^{2} = \ran E((0,1]) \subseteq \dom G_{+}^{\nicefrac{1}{2}} = \ran \iota_{+} = D_{+}$, as $(0,1]$ is a bounded set. We can write every $f \in D_{+}$ as
          \begin{equation*}
            f = E((1,\infty))f + E((0,1])f
          \end{equation*}
          and since $E((0,1]) f \in D_{+}$, we conclude that also $E((1,\infty))f \in D_{+}$. For an arbitrary $f \in \hsmid^{1} \subseteq \hsmid$ there exists a sequence $(f_{n})_{n\in\N}$ in $D_{+}$ such that $f_{n} \to f$ w.r.t.\ $\norm{\cdot}_{\hsmid}$. Since also $(E((1,\infty))f_{n})_{n\in\N}$ converges to $E((1,\infty))f = f$ by continuity, and $E((1,\infty))f \in D_{+}$, we conclude that $\hsmid^{1} \cap D_{+}$ is dense in $\hsmid^{1}$ (w.r.t.\ $\norm{\cdot}_{\hsmid}$). On the other hand, $\hsmid^{2} \subseteq D_{+}$.

    \item \emph{Decomposition of $\hs$.}
          For $f \in D_{+}$ we have
          \begin{multline*}
            \norm{E((0,1]) f}_{\hs}^{2} = \norm{G_{+}^{\nicefrac{1}{2}} E((0,1]) f}_{\hsmid}^{2}
            \\ = \int_{(0,1]} \abs{\lambda}^{2} \dx[E_{f,f}]
            \leq \int_{(0,\infty)} \abs{\lambda}^{2} \dx[E_{f,f}] = \norm{f}_{\hs}^{2},
          \end{multline*}
          and
          \begin{multline*}
          \norm{E((1,\infty)) f}_{\hs}^{2} = \norm{G_{+}^{\nicefrac{1}{2}} E((1,\infty)) f}_{\hsmid}^{2}
          \\= \int_{(1,\infty)} \abs{\lambda}^{2} \dx[E_{f,f}]
          \leq\int _{(0,\infty)} \abs{\lambda}^{2} \dx[E_{f,f}] = \norm{f}_{\hs}^{2}.
          \end{multline*}
          Hence, the spectral projections $E((0,1])$ and $E((1,\infty))$ are also continuous on $D_{+}$ with respect to $\norm{\cdot}_{\hs}$ and we can extend these projections continuously on $\hs$. Note that for $f \in D_{+}$ we have $G^{\nicefrac{1}{2}} E(\Delta) f = E(\Delta) G^{\nicefrac{1}{2}}$ for all $\Delta$ in the Borel sets of $\R$. Hence, we have for $f,g \in D_{+}$
          \begin{multline*}
            \scprod{E((0,1])f}{E((1,\infty))g}_{\hs}
            = \scprod{G_{+}^{\nicefrac{1}{2}}E((0,1])f}{G_{+}^{\nicefrac{1}{2}}E((1,\infty))g}_{\hsmid} \\
            = \vphantom{\underbrace{E}_{=0}}
            \scprod{G_{+}\smash[b]{\underbrace{E((1,\infty)) E((0,1])}_{=\mathrlap{0}}}f}{g}_{\hsmid} = 0,
          \end{multline*}
          which implies that the extensions of $E((0,1])\big\vert_{D_{+}}$ and $E((1,\infty))\big\vert_{D_{+}}$ are orthogonal projections on $\hs$.
          Moreover, for $f \in \hs$ there exists a sequence $(f_{n})_{n\in\N}$ in $D_{+}$ that converges to $f$ w.r.t.\ $\norm{\cdot}_{\hs}$. By the continuity of projections we conclude that $(E((0,1]) f_{n})_{n\in\N}$ and $(E((1,\infty))f_{n})_{n\in\N}$ converges and therefore
          \begin{multline*}
            f = \lim_{n\to\infty} f_{n} = \lim_{n\to\infty} E((0,1]) f_{n} + E((1,\infty)) f_{n} \\
            = \lim_{n\to\infty} E((0,1]) f_{n} + \lim_{n\to\infty}E((1,\infty)) f_{n}.
          \end{multline*}
          This leads to: the extensions of these projections are also complementary. We denote these extensions by $E((0,1])_{+}$ and $E((1,\infty))_{+}$ and we have
          \begin{equation*}
            \hs =
            \underbrace{\ran E((1,\infty))_{+}}_{\eqqcolon \mathrlap{\hs^{1}}}
            \mathclose{} \oplus \mathopen{}
            \underbrace{\ran E((0,1])_{+}}_{\eqqcolon \mathrlap{\hsdual^{2}}}.
          \end{equation*}

    \item \emph{Relationship between the decompositions of $\hsmid$ and $\hs$.}
          Note that $E((1,\infty))_{+} D_{+} = E((1,\infty))D_{+} = \hsmid^{1} \cap D_{+}$. Furthermore, for $f \in \hsmid^{1} \cap D_{+}$ we have
          \begin{multline}\label{eq:hs-norm-stronger-than-hsmid-norm}
            \norm{f}_{\hs}^{2} = \norm{E((1,\infty)) f}_{\hs}^{2} = \norm{G_{+}^{\nicefrac{1}{2}} E((1,\infty)) f}_{\hsmid}^{2} \\
            = \int_{(1,\infty)} \abs{\lambda}^{2} \dx[E_{f,f}] \geq \inf_{\lambda \in (1,\infty)} \abs{\lambda}^{2} \norm{f}_{\hsmid}^{2} \geq \norm{f}_{\hsmid}^{2}.
          \end{multline}
          Now for $f \in \hs^{1}$ there exists a sequence $(f_{n})_{n\in\N}$ in $D_{+}$ that converges to $f$ w.r.t.\ $\norm{\cdot}_{\hs}$ and therefore also $(\tilde{f}_{n})_{n\in\N} = (E((1,\infty))_{+}f_{n})_{n\in\N}$ converges to $f$ w.r.t.\ $\norm{\cdot}_{\hs}$. By \eqref{eq:hs-norm-stronger-than-hsmid-norm} we have
          \begin{align*}
            \norm{\tilde{f}_{n} - \tilde{f}_{m}}_{\hsmid} \leq \norm{\tilde{f}_{n} - \tilde{f}_{m}}_{\hs} \to 0,
          \end{align*}
          which implies that $(\tilde{f}_{n})_{n\in\N}$ is a Cauchy sequence in $\hsmid^{1}$ (w.r.t.\ $\norm{\cdot}_{\hsmid}$). By the closedness of $\iota_{+}$ the limit of this sequence (w.r.t.\ $\norm{\cdot}_{\hsmid}$) has to coincide with $f$. Hence, $\hs^{1} = \hsmid^{1} \cap D_{+}$ and the restricted embedding $\iota_{+}\big\vert_{\hs^{1}}\colon \hs^{1} \to \hsmid^{1}$ is continuous.

          On the other hand, since $\hsmid^{2} \subseteq D_{+}$ we automatically have $\hsmid^{2} \subseteq \hsdual^{2}$, by construction. Furthermore, for $f \in \hsmid^{2}$ we have
          \begin{multline}\label{eq:hs-norm-weaker-than-hsmid-norm}
            \norm{f}_{\hs}^{2} = \norm{E((0,1]) f}_{\hs}^{2} = \norm{G_{+}^{\nicefrac{1}{2}} E((0,1]) f}_{\hsmid}^{2} \\
            = \int_{(0,1]} \abs{\lambda}^{2} \dx[E_{f,f}] \leq \sup_{\lambda \in (0,1]} \abs{\lambda}^{2} \norm{f}_{\hsmid}^{2} \leq \norm{f}_{\hsmid}^{2}.
          \end{multline}
          This implies that the inverse embedding $\iota_{+}^{-1}$ restricted to $\hsmid^{2}$ is continuous, i.e., $\iota_{+}^{-1}\big\vert_{\hsmid^{2}}\colon \hsmid^{2} \to \hsdual^{2}$ is continuous. Hence, we have
          \begin{align*}
            \hsmid^{2} \subseteq \hsdual^{2}\quad\text{and}\quad \hs^{1} \subseteq \hsmid^{1}
          \end{align*}
          densely with continuous embeddings

    \item \emph{Decomposition of $\hsdual$.}
          Note that for $g \in D_{-}$ we have
          \[
          \norm{g}_{\hsdual} = \norm{G_{-}^{\nicefrac{1}{2}}g}_{\hsmid} = \norm{G_{+}^{-\nicefrac{1}{2}}g}_{\hsmid}
          \]
          and additionally by the rules for the spectral calculus we have
          \[
          G_{-}^{\nicefrac{1}{2}} = G_{+}^{-\nicefrac{1}{2}} = \int_{(0,\infty)} \frac{1}{\lambda} \dx[E].
          \]
          Hence, the exact same construction as in the second step (replace $\hs$ by $\hsdual$, $D_{+}$ by $D_{-}$, $G_{+}$ by $G_{-}$ and $\abs{\lambda}$ by $\abs{\frac{1}{\lambda}}$) gives the decomposition
          \begin{equation*}
            \hsdual =
            \underbrace{\ran E((1,\infty))_{-}}_{\eqqcolon\mathrlap{\hsdual^{1}}}
            \mathclose{} \oplus \mathopen{}
            \underbrace{\ran E((0,1])_{-}}_{\eqqcolon\mathrlap{\hs^{2}}}.
          \end{equation*}

    \item \emph{Relation ship between the decompositions of $\hsmid$ and $\hsdual$.}
          Again repeating the arguments of the third step.
          In particular, for $g \in D_{-}$ we have
          \begin{multline*}
            \norm{E((0,1])g}_{\hsdual}^{2} = \norm{G_{-}^{\nicefrac{1}{2}}E((0,1])g}_{\hsmid}^{2} \\
            = \int_{(0,1]} \abs*{\frac{1}{\lambda}}^{2} \dx[E_{g,g}] \geq \inf_{\lambda \in (0,1]} \abs*{\frac{1}{\lambda}}^{2} \norm{g}_{\hsmid}^{2} = \norm{g}_{\hsmid}^{2}
          \end{multline*}
          and
          \begin{multline*}
            \norm{E((1,\infty))g}_{\hsdual}^{2} = \norm{G_{-}^{\nicefrac{1}{2}}E((1,\infty))g}_{\hsmid}^{2} \\
            = \int_{(1,\infty)} \abs*{\frac{1}{\lambda}}^{2} \dx[E_{g,g}] \leq \inf_{\lambda \in (1,\infty)} \abs*{\frac{1}{\lambda}}^{2} \norm{g}_{\hsmid}^{2} = \norm{g}_{\hsmid}^{2}.
          \end{multline*}
          This implies
          \(
          \iota_{-}\big\vert_{\hs^{2}} \colon \hs^{2} \to \hsmid^{2}
          \)
          and
          \(
          \iota_{-}^{-1}\big\vert_{\hsmid^{1}} \colon \hsmid^{1} \to \hsdual^{1}
          \)
          are continuous. In particular, we have
          \begin{equation*}
            \hs^{2} \subseteq \hsmid^{2}
            \quad\text{and}\quad
            \hsmid^{1} \subseteq \hsdual^{1}
          \end{equation*}
          densely with continuous embeddings.

    \item \emph{Dualities.}
          By Hahn-Banach we can identify $(\hsdual^{2})\dual$ with $\hs\dual\big\vert_{\hsdual^{2}}$.
          Moreover, for $f \in \hs$ and $g \in \hsdual$ there exist sequences $(f_{n})_{n\in\N}$ in $D_{+}$ and $(g_{n})_{n\in\N}$ in $D_{-}$ such that $f_{n} \to f$ w.r.t.\ $\norm{\cdot}_{\hs}$ and $g_{n} \to g$ w.r.t.\ $\norm{\cdot}_{\hsdual}$. Hence,
          \begin{align*}
            \dualprod{E((0,1])_{+}f}{E((1,\infty))_{-}g}_{\hs,\hsdual}
            &= \lim_{n\to\infty}\dualprod{E((0,1])_{+}f_{n}}{E((1,\infty))_{-}g_{n}}_{\hs,\hsdual} \\
            &= \lim_{n\to\infty} \scprod{E((0,1])f_{n}}{E((1,\infty))g_{n}}_{\hsmid} = 0.
          \end{align*}
          Clearly, we also have $\dualprod{E((1,\infty))_{+}f}{E((0,1])_{-}g}_{\hs,\hsdual} = 0$.
          For $\phi \in (\hsdual^{2})\dual$ there exists a $g \in \hsdual$ such that
          \begin{equation*}
            \phi(f) = \dualprod{g}{f}_{\hsdual,\hs}
            = \dualprod{E((0,1])_{-}g}{f}_{\hsdual,\hs} + \underbrace{\dualprod{E((1,\infty))_{-}g}{f}_{\hsdual,\hs}}_{=\mathrlap{0}}
            \quad\forall
            f \in \hsdual^{2}.
          \end{equation*}
          Moreover,
          \begin{align*}
            \norm{\phi}_{(\hsdual^{2})\dual}
            &= \sup_{f \in \hsdual^{2}\setminus\sset{0}} \frac{\abs{\phi(f)}}{\norm{f}_{\hs}}
              = \sup_{f \in \hsdual^{2}\setminus\sset{0}} \frac{\abs{\dualprod{E((0,1])_{-}g}{f}_{\hsdual,\hs}}}{\norm{f}_{\hs}} \\
            &= \sup_{f \in \hs\setminus\sset{0}} \frac{\abs{\dualprod{E((0,1])_{-}g}{f}_{\hsdual,\hs}}}{\norm{f}_{\hs}}
              = \norm{E((0,1])_{-}g}_{\hsdual}
          \end{align*}
          On the other hand, if $\dualprod{E((0,1])_{-}g}{f}_{\hsdual,\hs} = 0$ for all $f \in \hsdual^{2}$, then we automatically have $\dualprod{E((0,1])_{-}g}{f}_{\hsdual,\hs} = 0$ for all $f \in \hs$ and therefore $E((0,1])_{-}g = 0$.
          In conclusion $(\hsdual^{2},\hs^{2})$ is a complete dual pair and $(\hsdual^{2},\hsmid^{2},\hs^{2})$ is a quasi Gelfand triple with the embeddings $\iota_{+}\big\vert_{\hsdual^{2}}$ and $\iota_{-}\big\vert_{\hs^{2}}$. Moreover, since $\iota_{-}\big\vert_{\hs^{2}}$ is continuous, it is even an ``ordinary'' Gelfand triple ($\hs^{2} \subseteq \hsmid^{2} \subseteq \hsdual^{2}$).

          We can show completely analogously that also $(\hs^{1},\hsmid^{1},\hsdual^{1})$ is an ``ordinary'' Gelfand triple ($\hs^{1} \subseteq \hsmid^{1} \subseteq \hsdual^{1}$). \qedhere
  \end{enumerate}
\end{proof}

Note that this decomposition is not unique as we could have split the space $\hsmid$ by any two subspaces $\ran E(\Delta)$ and $\ran E(\Delta^{\complement})$, where $\Delta \subseteq \R_{+}$ is a bounded non-empty Borel set.

Finally, we end with two conjectures

\begin{conjecture}[weak]\label{con:weak}
  Every pre-quasi Gelfand triple of Hilbert spaces is a quasi Gelfand triple.
\end{conjecture}

\begin{conjecture}[strong]\label{con:strong}
  Every pre-quasi Gelfand triple is a quasi Gelfand triple.
\end{conjecture}

At least the weak conjecture seems to be true, but all attempts failed so far. In fact \Cref{th:Zm-is-an-initial-object} and \Cref{th:decomposition-into-ordinary-Gelfand-triple} are the result of failed attempts to prove the weak conjecture. The strong conjecture seems much more difficult, as a lot of Hilbert space theory is unavailable.

A positive answer to (at least) the weak conjecture would automatically answer the question whether the weak and strong definition of boundary trace operators for differential operators coincide.

\section*{Conclusion}

We have introduces a generalization of Gelfand triple that does not need continuous embeddings. This was done by replacing the continuity of the embeddings by closedness. We showed that $D_{+} \cap D_{-}$, the set that is in the intersection of the quasi Gelfand triple, is dense in the pivot space $\hsmid$.

If we regard quasi Gelfand triples of Hilbert spaces, then we can show that $D_{+} \cap D_{-}$ is also dense in $\hs$ and $\hsdual$ w.r.t.\ their norms. Furthermore, we have shown that there exists a smallest space were we can embed the entire quasi Gelfand triple structure preservingly.

Finally, we have shown that every quasi Gelfand triple is associated to a Gram operator and the other way round. This led us to a decomposition of the quasi Gelfand triple into two ``ordinary'' Gelfand triples.

We ended with the weak and strong version of the conjecture that every pre-quasi Gelfand triple is in fact already a quasi Gelfand triple.

One application that we did not cover, that is still worth mentioning:
Quasi Gelfand triples can be used to properly define boundary spaces and characterizing suitable boundary conditions for partial differential equations that lead to existence and uniqueness of solutions, see~\cite{Sk21}.

\appendix
\section{Auxiliary Results}

\begin{lemma}\label{th:weak-convergent-bounded}
  Let $(x_n)_{n\in\N}$ be a sequence in a normed vector space $X$ that converges w.r.t.\ the weak topology to an $x_0\in X$. Then $(x_n)_{n\in\N}$ is bounded, i.e., $\sup_{n\in\N} \norm{x_n}_{X} < +\infty$.
\end{lemma}

\begin{proof}
  \newcommand{\canEm}{\iota} 
  Let $\canEm$ denote the canonical embedding from $X$ into $X\dual[2]$ that maps $x$ to $\dualprod{x}{\cdot}_{X,X\dual}$. Then, by assumption, for every fixed $\phi\in X\dual$ $(\canEm x_n)(\phi) \to (\canEm x_0)(\phi)$, in particular $\sup_{n\in\N} \abs{(\canEm x_n)(\phi)} < \infty$. The principle of uniform boundedness yields $\sup_{n\in\N} \norm{\canEm x_n}_{X\dual[2]} < +\infty$. Since $\norm{\canEm x}_{X\dual[2]} = \norm{x}_{X}$ for every $x\in X$, this proves the assertion.
\end{proof}

\begin{lemma}\label{th:weak-to-strong-convergent}
  Let $(x_n)_{n\in\N}$ be a weak convergent sequence in a Hilbert space $H$ with limit $x$. Then there exists a subsequence $(x_{n(k)})_{k\in\N}$ such that
  \[
    \norm[\bigg]{\frac{1}{N}\sum_{k=1}^{N} x_{n(k)} - x} \to 0.
  \]
\end{lemma}

\begin{proof}
  We assume that $x = 0$. For the general result we just need to replace $x_n$ by $x_n - x$.

  We define the subsequence inductively: $n(1) = 1$ and for $k > 1$ we choose $n(k)$ such that
  \[
    \abs{\scprod{x_{n(k)}}{x_{n(j)}}} \leq \frac{1}{k}
    \quad \text{for all} \quad j < k.
  \]
  This is possible, because $(x_n)_{n\in\N}$ converges weakly to $0$.
  Hence, by~\Cref{th:weak-convergent-bounded}
  $\sup_{n\in\N}\norm{x_{n}} \leq C$. This yields
  \begin{align*}
    \norm[\bigg]{\frac{1}{N}\sum_{k=1}^{N} x_{n(k)}}^{2} &= \frac{1}{N^2}\sum_{k=1}^{N}\sum_{j=1}^{N}\scprod{x_{n(k)}}{x_{n(j)}} \\
    &= \frac{1}{N^{2}} \sum_{k=1}^{N} \norm{x_{n(k)}}^{2} + \frac{1}{N^{2}}\sum_{j=1}^{N} \sum_{k=j+1}^{N} 2\Re\scprod{x_{n(k)}}{x_{n(j)}} \\
    &\leq \frac{1}{N} C^{2} +  \frac{2}{N^{2}} \sum_{j=1}^{N} \sum_{k=j+1}^{N} \frac{1}{k} \leq \frac{C^{2}}{N} + \frac{1}{N}\ln(N) \to 0.
    \ifSn{\tag*{\qedhere}}{\qedhere}
  \end{align*}
\end{proof}


The next lemma is also true for general linear relations. However, since densely defined linear operators are enough for our purpose we restrict ourselves to these operators, also to use commonly known techniques.

\begin{lemma}\label{le:adjoints-for-different-dualities}
  Let $(X_{1},Y_{1})$, $(X_{1},Z_{1})$, $(X_{2},Y_{2})$ and $(X_{2},Z_{2})$ be dual pairs and $\Psi_{1}\colon Y_{1} \to Z_{1}$ and $\Psi_{2}\colon Y_{2} \to Z_{2}$ be the isomorphisms between $Y_{1}$ and $Z_{1}$, and $Y_{2}$ and $Z_{2}$, respectively. Then for a densely defined linear operator $A$ from $X_{1}$ to $X_{2}$ we have
  \begin{equation*}
    A\adjunX{Z_{2}\times Z_{1}}
    = \Psi_{1} A\adjunX{Y_{2}\times Y_{1}} \Psi_{2}^{-1}.
  \end{equation*}
\end{lemma}

\begin{figure}[ht]
  \centering
  \colorlet{mycolor}{white!60!black}
  \colorlet{mycolorii}{blue}
  \colorlet{mycoloriii}{green!40!black}
  \begin{tikzcd}[column sep=5em,row sep=2em]
    Z_{1}\ar[mycolorii,bend right=45,outer sep=-1pt,<-]{dd}{\Psi_{1}} & \ar[mycoloriii,swap]{l}{A\adjunX{Z_{2} \times Z_{1}}}Z_{2}\ar[bend left=45,outer sep=-1pt,swap,<-]{dd}{\Psi_{2}} \\
    \textcolor{mycolor}{X_{1}}\ar[latex-latex,dotted]{u}{} \ar[latex-latex,dotted]{d}{} \ar[mycolor]{r}{A} & \textcolor{mycolor}{X_{2}}\ar[latex-latex,dotted]{u}{} \ar[latex-latex,dotted]{d} \\
    Y_{1}\ar[bend left=45,shift left=5pt,<-]{uu}{\Psi_{1}^{-1}} & \ar[mycolorii,swap]{l}{A\adjunX{Y_{2} \times Y_{1}}}Y_{2} \ar[mycolorii,bend right=45,shift right=5pt,swap,<-]{uu}{\Psi_{2}^{-1}}
  \end{tikzcd}
  \caption{\label{fig:adjoints-for-different-dual-paris}$\textcolor{mycoloriii}{A\adjunX{Z_{2}\times Z_{1}}} = \textcolor{mycolorii}{\Psi_{1} A\adjunX{Y_{2}\times Y_{1}} \Psi_{2}^{-1}}$}
\end{figure}

\begin{proof}
  Let $z_{2} \in Z_{2}$ be such that $\Psi_{2}^{-1}z \in \dom A\adjunX{Y_{2} \times Y_{1}}$. Then
  \begin{align*}
    \dualprod{Ax_{1}}{z_{2}}_{X_{2},Z_{2}}
    &= \dualprod{Ax_{1}}{\Psi_{2}^{-1}z_{2}}_{X_{2},Y_{2}}
      = \dualprod{Ax_{1}}{\Psi_{2}^{-1} z_{2}}_{X_{2},Y_{2}}
      \ifBirk{\\ &}{}
                   = \dualprod{x_{1}}{A\adjunX{Y_{2}\times Y_{1}} \Psi_{2}^{-1} z_{2}}_{X_{1},Y_{1}} \\
    &= \dualprod{x_{1}}{\Psi_{1} A\adjunX{Y_{2}\times Y_{1}} \Psi_{2}^{-1} z_{2}}_{X_{1},Z_{1}}.
  \end{align*}
  This implies $\Psi_{1} A\adjunX{Y_{2}\times Y_{1}} \Psi_{2}^{-1} \subseteq A\adjunX{Z_{2} \times Z_{1}}$. The same steps with $Z_{2}$ and $Z_{1}$ replaced with $Y_{2}$ and $Y_{1}$ yield the reversed inclusion.
\end{proof}


The following theorem can be found in~\cite[Th.~2 p.~200]{fana-yosida}, we just changed that the operator maps into a different space, which does not change the proof.

\begin{theorem}[J. von Neumann]\label{th:TTadjun-self-adjoint}
Let $T$ be a closed linear operator from the Hilbert space $X$ to the Hilbert space $Y$. Then $T\adjun T$ and $TT\adjun$ are self-adjoint, and $(\opid_{X} + T\adjun T)$ and $(\opid_{Y} + TT\adjun)$ are boundedly invertible.
\end{theorem}

Note that here the adjoint $T\adjun$ is calculated with respect to the ``natural'' dual pairs $(X,X)$ and $(Y,Y)$, i.e., $T\adjun = T\adjunX{Y\times X}$.

\begin{proof}
  Since
  $T\adjun =
  \begin{bsmallmatrix}
    0 & \opid_{Y} \\
    -\opid_{X} & 0
  \end{bsmallmatrix} T^{\perp}$,
  we have
  $T \oplus
  \begin{bsmallmatrix}
    0 & -\opid_{X} \\
    \opid_{Y} & 0
  \end{bsmallmatrix}
  T\adjun = X\times Y$.
  Hence, for
  $
  \begin{bsmallmatrix}
    h \\ 0
  \end{bsmallmatrix}
  \in X\times Y$ there are unique $x \in \dom T$ and $y \in \dom T\adjun$ such that
  \begin{equation}\label{eq:decomposition}
    \begin{bmatrix}
      h \\ 0
    \end{bmatrix}
    =
    \begin{bmatrix}
      x \\ Tx
    \end{bmatrix}
    +
    \begin{bmatrix}
      -T\adjun y \\ y
    \end{bmatrix}.
  \end{equation}
  Consequently, $h=x - T\adjun y$ and $y = -Tx$, which implies $x \in \dom T\adjun T$ and
  \begin{align*}
    h = x + T\adjun T x.
  \end{align*}
  Because of the uniqueness of the decomposition in~\eqref{eq:decomposition}, $x \in \dom T\adjun T$ is uniquely determined by $h \in X$. Therefore, $(\opid_{X} + T\adjun T)^{-1}$ is a well-defined and everywhere defined operator.

  For $h_{1},h_{2} \in X$, we define $x_{1} \coloneqq (\opid_{X} + T\adjun T)^{-1} h_{1}$ and $x_{2} \coloneqq (\opid_{X} + T\adjun T)^{-1} h_{2}$. Then $x_{1}, x_{2} \in \dom T\adjun T$ and, by the closedness  of $T$, $T\adjun[2] = T$. Hence,
  \begin{align*}
    \scprod{h_{1}}{ (\opid_{X} + T\adjun T)^{-1} h_{2}}
    &= \scprod{(\opid_{X} + T\adjun T) x_{1}}{ x_{2}} = \scprod{x_{1}}{x_{2}} + \scprod{T\adjun T x_{1}}{ x_{2}} \\
    &= \scprod{x_{1}}{x_{2}} + \scprod{T x_{1}}{T x_{2}}= \scprod{x_{1}}{x_{2}} + \scprod{x_{1}}{ T\adjun T x_{2}} \\
    &= \scprod{x_{1}}{ (\opid_{X} + T\adjun T) x_{2}} = \scprod{(\opid_{X} + T\adjun T)^{-1} h_{1}}{ h_{2}},
  \end{align*}
  which yields that $(\opid_{X} + T\adjun T)^{-1}$ is self-adjoint. Therefore $(\opid_{X} + T\adjun T)$ and $T\adjun T$ are also self-adjoint. Moreover, $(\opid_{X} + T\adjun T)^{-1}$ is bounded as a closed and everywhere defined operator.

  By $TT\adjun = (T\adjun)\adjun(T\adjun)$ the other statements follow by the already shown.
\end{proof}

\begin{lemma}
  Let $T$ be the operator from the previous theorem. Then $\dom T\adjun T$ is a core of $T$.
\end{lemma}

\begin{proof}
  Note that $\dom T\adjun T$ is a core of $T$ is equivalent that to $\dom T\adjun T$ is dense in $\dom T$ with respect to the graph norm. Hence, it is sufficient to show that the orthogonal complement of $\dom T\adjun T$ is $\sset{0}$ w.r.t.\ the graph inner product.
  Suppose $\dom T\adjun T$ is not a core, then there exists an $x \in \dom T\setminus\sset{0}$ such that
  \begin{equation*}
    0 = \scprod{x}{y}_{T} = \scprod{x}{y}_{X} + \scprod{Tx}{Ty}_{Y} = \scprod{x}{y + T\adjun T y}_{X}\quad\text{for all}\quad y \in \dom T\adjun T.
  \end{equation*}
  By \Cref{th:TTadjun-self-adjoint} $(\opid + T\adjun T)y$ is surjective, which implies $x=0$ and contradicts the assumption.
\end{proof}

In the next proposition we will look at the situation where we deal with Hilbert spaces, but work with another dual pair. We will denote the adjoint with respect to the canonical Hilbert space dual pair by $\hadjunsymb$ and the adjoint with respect to the other dual pair by $\dadjunsymb$.

\begin{proposition}\label{th:TadjunT-selfadjoint-for-dual-pair}
  Let $X$, $H$ be Hilbert spaces, $(X,Y)$ be a complete dual pair and $T\colon \dom T \subseteq X \to H$ be a densely defined and closed linear operator. Then $T\dadjun T\colon \dom T\dadjun T \subseteq X \to Y$ is self-adjoint, i.e. $(T\dadjun T)\dadjun = T\dadjun T$.
  Moreover, $\dom T\dadjun T$ is a core of $T$.
\end{proposition}

\begin{proof}
  For $x,y \in \dom T\dadjun T$ we have
  \begin{align*}
    \dualprod{T\dadjun Tx}{y}_{Y,X} = \scprod{Tx}{Ty}_{H} = \dualprod{x}{T\dadjun Ty}_{X,Y},
  \end{align*}
  which leads to $T\dadjun T \subseteq (T\dadjun T)\dadjun$.

  By \Cref{th:TTadjun-self-adjoint} we already know that $T\hadjun T$ is self-adjoint. Let $\Psi\colon X \to Y$ the duality mapping, i.e.,
  \(
  \dualprod{\Psi x}{y}_{Y,X} = \scprod{x}{y}_{X}
  \)
  for $x,y \in X$.
  Then $T\dadjun = \Psi T\hadjun$ (by \Cref{le:adjoints-for-different-dualities}) and therefore $T\dadjun T = \Psi T\hadjun T$. Now for $x \in \dom (T\dadjun T)\dadjun$ and $y \in \dom T\dadjun T = \dom T\hadjun T$ we have
  \begin{align*}
    \scprod{\Psi^{-1}(T\dadjun T)\dadjun x}{y}_{X}
    &= \dualprod{(T\dadjun T)\dadjun x}{y}_{Y,X}
    = \dualprod{x}{T\dadjun Ty}_{X,Y}
    = \scprod{x}{\underbrace{\Psi^{-1} T\dadjun}_{=\mathrlap{T\hadjun}} Ty}_{X} \\
    &= \scprod{x}{T\hadjun Ty}_{X}.
  \end{align*}
  This implies $\Psi^{-1}(T\dadjun T)\dadjun \subseteq (T\hadjun T)\hadjun = T\hadjun T$ and applying $\Psi$ on both sides gives $(T\dadjun T)\dadjun \subseteq  \Psi T\hadjun T = T\dadjun T$.

  The last assertion follows from $\dom T\dadjun T = \dom T\hadjun T$ and $\dom T\hadjun T$ is a core of $T$,
\end{proof}



\end{document}